\documentclass[12pt]{article}
\usepackage{graphicx}
\usepackage{amssymb}
\usepackage{latexsym, amssymb, amsfonts, amsmath, amsthm, graphics, graphicx, subfigure, bbm, stmaryrd, mathrsfs}
\usepackage{epstopdf}

\usepackage{authblk}

\usepackage[margin=2cm]{geometry}

\usepackage{latexsym, amssymb, amsfonts, amsmath, amsthm, graphics, graphicx, subfigure, bbm, stmaryrd, xcolor}
\usepackage{epstopdf}

\allowdisplaybreaks[4]

\topmargin by -\headsep \textheight 8.9in \oddsidemargin 0pt
\evensidemargin \oddsidemargin \marginparwidth 0.5in \textwidth
6.5in

\newcommand{\ZZ}{{\mathbb Z}}

\newcommand{\PP}{{\mathbb P}}

\newtheorem{theorem}{Theorem}[section]
\newtheorem{proposition}[theorem]{Proposition}
\newtheorem{corollary}[theorem]{Corollary}
\newtheorem{lemma}[theorem]{Lemma}
\newtheorem{definition}[theorem]{Definition}

\newtheorem{remark}[theorem]{Remark}

\numberwithin{equation}{subsection}

\makeatletter
\renewcommand\section{\@startsection {section}{1}{\z@}%
    {-3.5ex \@plus -1ex \@minus -.2ex}%
    {2.3ex \@plus.2ex}%
    {\normalfont\fontsize{18}{19}\bfseries}}
\makeatother

\makeatletter
\renewcommand\subsection{\@startsection {subsection}{1}{\z@}%
    {-1.5ex \@plus -1ex \@minus -.2ex}%
    {1.3ex \@plus.2ex}%
    {\normalfont\fontsize{13}{14}\bfseries}}
\makeatother

\makeatletter
\newcommand\xleftrightarrow[2][]{%
  \ext@arrow 9999{\longleftrightarrowfill@}{#1}{#2}}
\newcommand\longleftrightarrowfill@{%
  \arrowfill@\leftarrow\relbar\rightarrow}
\makeatother

\title{\bf Random nearest neighbor graphs: the translation invariant case}

\author[1]{Bounghun Bock$^1$, Michael Damron}
\author[2]{Jack Hanson}
\affil[1]{Georgia Tech}
\affil[2]{City University of New York, City College and Graduate Center}
\date{}

\begin{document}

\baselineskip20pt

\maketitle

\begin{abstract}
If $(\omega(e))$ is a family of random variables (weights) assigned to the edges of $\mathbb{Z}^d$, the nearest neighbor graph is the directed graph induced by all edges $\langle x,y \rangle$ such that $\omega(\{x,y\})$ is minimal among all neighbors $y$ of $x$. That is, each vertex points to its closest neighbor, if the weights are viewed as edge-lengths. Nanda-Newman introduced nearest neighbor graphs when the weights are i.i.d.~and continuously distributed and proved that a.s., all components of the undirected version of the graph are finite. We study the case of translation invariant, distinct weights, and prove that nearest neighbor graphs do not contain doubly-infinite directed paths. In contrast to the i.i.d.~case, we show that in this stationary case, the graphs can contain either one or two infinite components (but not more) in dimension two, and $k$ infinite components for any $k \in [1,\infty]$ in dimension $\geq 3$. The latter constructions use a general procedure to exhibit a certain class of directed graphs as nearest neighbor graphs with distinct weights, and thereby characterize all translation invariant nearest neighbor graphs. We also discuss relations to geodesic graphs from first-passage percolation and implications for the coalescing walk model of Chaika-Krishnan.
\end{abstract}

%
%
%
%
%
%

\section{Introduction}
Random nearest neighbor graphs were introduced by Nanda-Newman \cite{MR1716765} in the context of the cubic lattice $\mathbb{Z}^d$, but we will define them on general graphs. The directed nearest neighbor graph $\mathcal{N}_D$ is defined on a given graph $G = (V,E)$ using real-valued edge-weights $(\omega(e))_{e \in E}$. We will assume that $G$ does not have self-loops (it has no edges of the form $\{v,v\}$ for $v \in V$) and does not have multiple edges between any two vertices (so that each edge is uniquely identified by its endpoints). The vertex set of $\mathcal{N}_D$ is $V$ and the edge set is the set
\[
\left\{ \langle x,y \rangle : \{x,y\} \in E \text{ and } \omega(\{x,y\}) \leq \omega(\{x,z\}) \text{ for all }z \text{ with } \{x,z\} \in E\right\}.
\]
That is, each vertex points to its neighbors $y$ which minimize the weight $\omega(\{x,y\})$. If some vertex $x$ has infinite degree, and no neighbor minimizes $\omega$, then $x$ does not point to any neighbor. The undirected nearest neighbor graph $\mathcal{N}$ is the undirected version of $\mathcal{N}_D$, with vertex set $V$ and edge set $\{\{x,y\} : \langle x, y \rangle \text{ is an edge of }\mathcal{N}_D\}$. Note that if a vertex has finite degree in $G$, it has out-degree at least one in $\mathcal{N}_D$ and if the weights $(\omega(e))$ are all distinct, any vertex has out-degree at most one in $\mathcal{N}_D$.


\subsection{Background}

In  \cite{MR1716765}, Nanda-Newman studied nearest neighbor graphs on $\mathbb{Z}^d$ in the case that the weights are i.i.d.~with a common uniform $(0,1)$ distribution. (Here, a.s.~each vertex has out-degree exactly one in $\mathcal{N}_D$.) One of their main results was that a.s., $\mathcal{N}$ has only finite components and, further, that connection probabilities decay rapidly: 
\[
\mathbb{P}\left(0 \text{ and } x \text{ are in the same component of } \mathcal{N}\right) \leq \frac{C^{\|x-y\|_1}}{\|x-y\|_1!}.
\]
\cite[Lemma~2.3]{MR1716765} also contains a complete description of all finite clusters of $\mathcal{N}$: if $\mathcal{C}$ is one of the connected components of $\mathcal{N}$, let $\mathcal{C}_D$ be the directed subgraph of $\mathcal{N}_D$ induced by the vertices of $\mathcal{C}$. Then a.s. for all $\mathcal{C}$,
\begin{enumerate}
\item $\mathcal{C}$ is a tree,
\item $\mathcal{C}_D$ contains exactly one ``miniloop'' between some vertices $x$ and $y$, and
\item every edge in $\mathcal{C}_D$ (besides $\langle x,y \rangle$ and $\langle y,x \rangle$) is directed toward both $x$ and $y$.
\end{enumerate}
Here, a ``miniloop'' is a directed circuit of length two (see the definition of a directed circuit in the next section). These results give a more-or-less complete description of i.i.d.~nearest neighbor graphs on $\mathbb{Z}^d$, and were used by Nanda-Newman to study ``influence graphs'' arising from energy minimization procedures in disordered Ising models.



In this paper, we study the structure of $\mathcal{N}$ and $\mathcal{N}_D$ under weaker assumptions on the weights: that they are translation invariant and a.s.~distinct. We first show in Theorem~\ref{theorem: structure of path in N} that under these general conditions, $\mathcal{N}_D$ cannot contain doubly-infinite directed paths. Furthermore, although these graphs in the i.i.d.~case contain only finite components, in Corollary~\ref{cor: main_result}, we prove that for dimension $d=2$, $\mathcal{N}$ can contain exactly one or exactly two infinite components, and for dimensions $d \geq 3$ and any $k \in [1,\infty]$, $\mathcal{N}$ can contain exactly $k$ infinite components. (Theorem~\ref{theorem: structure of path in N} immediately implies that for $d=1$ there are only finite components; see Remark~\ref{rem: one_dimension}.) These constructions follow from a general result, Theorem~\ref{thm: construction}, which shows that directed graphs with certain properties (all vertices have out-degree one, there are no directed cycles of length at least three, and there are no doubly-infinite directed paths) can be realized as nearest neighbor graphs with distinct weights. In contrast to the situation in dimension $d \geq 3$, for dimension two, we show in Theorem~\ref{thm: two_components} that $\mathcal{N}$ cannot have more than two infinite components. This result follows from a detailed analysis of the topological structure of infinite components in the plane.

Although random nearest neighbor graphs on $\mathbb{Z}^d$ appear only to have been studied by Nanda-Newman, similar graphs have appeared in the literature. For example, \cite{HarrisMeester} studies percolation properties of some neighbor graphs, and in the context of Poisson models \cite{Ballister, HaggstromMeester, Kozakova} and bipartite graphs \cite{Pittel}, several authors have studied nearest neighbor-type models. Furthermore, Chaika-Krishnan \cite{CK, CK2} have introduced models of stationary coalescing walks, and the related directed graphs share some of the features of our graphs. In Section~\ref{sec: arjun}, we will explain the implications our results have for their models.

One motivation for studying nearest neighbor models with translation invariant weights comes from geodesic graphs constructed in first-passage percolation \cite{AH,BDH20,DH14}. These are distributional limits of directed graphs whose edge sets are unions of point-to-hyperplane geodesics. It is known that geodesic graphs in any dimension $d \geq 2$ do not contain doubly-infinite paths, but each of their vertices has out-degree one and there are no directed cycles. Therefore they satisfy the conditions of Theorem~\ref{thm: construction} and can be realized as nearest neighbor graphs. It is an important question to determine the number of infinite components of geodesic graphs in general dimensions, It is known that there is only one component in two dimensions, but for $d\geq 3$ the number of infinite components is not even known to be 1 or infinity. The results of this paper show that there are nearest neighbor graphs in any dimension $d \geq 3$ with any number of infinite components, so any work on these questions for geodesic graphs must use more detailed properties of the percolation model.

%
%


 \subsection{Main Results}
 
Our probabilistic results will concern the cubic lattice $\left( \mathbb{Z}^d, \mathcal{E}^d\right)$ with translation invariant weights. (General graphs are considered below.) For this reason, our probability space will be the product space $\Omega = \mathbb{R}^{\mathcal{E}^d}$ for some $d \geq 1$ with the product Borel sigma-algebra. Our probability measure $\mathbb{P}$ will be assumed to satisfy the following conditions:

\medskip
\noindent
\textbf{Assumption A: \label{asd} } 

\begin{enumerate}
\item[{\bf A1.}] $\PP$ is translation invariant. That is, for any $z \in \ZZ^d$, $\PP = \PP \circ T_z^{-1}$, where $T_z$ is the translation by $z$: for $\omega \in \Omega$, 
\begin{equation}\label{eq: T_z}
T_z \omega  \in \Omega \text{ is given by } \left( T_z\omega\right)(e) = \omega(e+z),
\end{equation}
where $e+z = \{x+z,y+z\}$ if $e = \{x,y\}$.
\item[{\bf A2.}] For any distinct $e,f \in \mathcal{E}^d$, $\PP \left( \omega(e) =\omega(f) \right) =0$.
\end{enumerate}
 
Note that assumption {\bf A} holds if the weights are i.i.d.~with a continuous common distribution. Furthermore, under item {\bf A2}, a.s.~every vertex has out-degree exactly one in $\mathcal{N}_D$.
 
Our first result states that under assumption {\bf A}, $\mathcal{N}_D$ has no infinite backward paths. For its statement, if $x,y$ are vertices of a directed graph, we write $x \to y$ if there is a directed path from $x$ to $y$. We use the convention that $x \to x$ for any $x$, so the graph $C_x$ defined below always has at least one vertex.
\begin{theorem} \label{theorem: structure of path in N} 
Let $d \geq 1$ and $\mathbb{P}$ be a measure satisfying assumption {\bf A}. For any $x \in \mathbb{Z}^d$, write $C_x$ for the subgraph of $\mathcal{N}_D$ induced by the vertices $y$ such that $y \to x$ in $\mathcal{N}_D$. Then
\[
\text{a.s., } C_x \text{ is finite for all }x \in \mathbb{Z}^d.
\] 
\end{theorem}
\noindent
We will prove Theorem~\ref{theorem: structure of path in N} in Section~\ref{sec: second_proof}. Because it relies on the mass transport principle, the argument can be extended to more general graphs satisfying the unimodular condition (see \cite[Sec.~8.2]{lyonsperes}). 

\begin{remark}\label{rem: one_dimension}
By Theorem~\ref{theorem: structure of path in N}, in the case $d=1$, all components of $\mathcal{N}$ must be finite. Indeed, by translation invariance, a.s.~one of the following three must occur: (a) all edges in $\mathcal{N}_D$ point left, (b) all edges point right, or (c) there are infinitely many left-pointing and right-pointing edges (and each left-pointing edge has a right-pointing edge somewhere to its left and somewhere to its right). Cases (a) and (b) cannot occur by the theorem. In case (c), all components of $\mathcal{N}$ are finite.
\end{remark}

Theorem~\ref{theorem: structure of path in N} states that nearest neighbor graphs cannot have infinite backward paths in any dimension. It is natural then to ask whether they can have infinite components at all and, if so, then how many there can be. In the next result we show that in two dimensions, there can be at most two infinite components.
\begin{theorem}\label{thm: two_components}
Let $d=2$ and $\mathbb{P}$ be a measure satisfying assumption {\bf A}. A.s., $\mathcal{N}$ has at most two infinite components.
\end{theorem}
We prove Theorem~\ref{thm: two_components} in Section~\ref{sec: dimension_two_proof}. The argument gives more information than what is stated in the theorem. It shows that if there are two infinite components, their closures in $\mathbb{Z}^2$ must be topological half-planes possibly separated by infinitely many finite components. In Remark~\ref{rem: topology}, we show that such finite separating components need not exist and, if they do, either they can be isolated from each other or the union of their vertex sets can be an (infinite) topological strip. 
 
In the third result, we show that certain directed graphs (and therefore certain random graph models) can be realized as nearest neighbor graphs. In its statement, a directed cycle of length $\ell$ in a directed graph is a sequence of directed edges $\langle x_0,x_1\rangle, \langle x_1, x_2 \rangle, \dots, \langle x_{\ell-1},x_\ell \rangle$ such that $x_0, \dots, x_{\ell-1}$ are all distinct and $x_\ell = x_0$.
\begin{theorem}\label{thm: construction}
Let $G=(V,E)$ be a graph such that $E$ is countable and let $\mathbb{G} = (V,\mathbb{E})$ be a directed graph with the same vertex set $V$. Assume that
\begin{enumerate}
\item if $\langle x,y \rangle \in \mathbb{E}$ then $\{x,y\} \in E$,
\item each $x \in V$ has out-degree one in $\mathbb{G}$,
\item $\mathbb{G}$ has no directed cycles of length at least three, and
\item for each vertex $x \in V$, writing $C_x$ for the subgraph of $\mathbb{G}$ induced by $y \in V$ such that $y \to x$ in $\mathbb{G}$, $C_x$ is finite.
\end{enumerate}
There exists a collection $(\omega(e))_{e \in E}$ of distinct weights such that the nearest neighbor graph $\mathcal{N}_D$ corresponding to these weights is $\mathbb{G}$.
\end{theorem}
\noindent
The proof of Theorem~\ref{thm: construction} will be given in Section~\ref{sec: first_proof}.

\begin{remark}
Theorem~\ref{thm: construction} states that items 1-4 are sufficient for a given graph $\mathbb{G}$ to be a nearest neighbor graph. If the graph $\mathbb{G}$ is a random directed graph sampled from a translation invariant distribution, then the definition of the weights in \eqref{eq: omega_def} ensures that the corresponding $\mathbb{P}$ satisfies assumption {\bf A}. Conversely, the statement and proof of Theorem~\ref{theorem: structure of path in N} (see the enumerated properties of $\mathcal{N}$ in Section~\ref{sec: second_proof}) show that in the case where $G = \mathbb{Z}^d$ (or more generally, a graph satisfying the unimodular condition), these properties listed in items 1-4 are also necessary. In other words, these results characterize all nearest neighbor graphs under assumption {\bf A}.
\end{remark}


As a consequence, we can construct various different random nearest neighbor graphs on $\mathbb{Z}^d$ with $d \geq 2$ for measures satisfying {\bf A}. 
\begin{corollary} \label{cor: main_result}
\begin{enumerate}
\item Let $d=2$ and $k \in \{0,1,2\}$. There is a measure $\mathbb{P}$ satisfying {\bf A} such that a.s., $\mathcal{N}$ has exactly $k$ infinite components.
\item Let $d \geq 3$ and $k \in \{0, 1, 2, \dots\} \cup \{\infty\}$. There is a measure $\mathbb{P}$ satisfying {\bf A} such that a.s., $\mathcal{N}$ has exactly $k$ infinite components.
\end{enumerate}
\end{corollary}
\noindent
The proof of Corollary~\ref{cor: main_result} will be given in Section~\ref{sec: corollary}. Because nearest neighbor graphs with i.i.d., continuously distributed weights provide examples with $k=0$, we need only focus on the cases $k \geq 1$.

\subsection{Implications for coalescing walks}\label{sec: arjun}

    Our results, especially those of Theorem \ref{thm: two_components} in the  $\mathbb{Z}^2$ setting, relate to some of those of Chaika-Krishnan \cite{CK,CK2}. The ``stationary coalescing walk'' model studied there amounts to a directed random graph $\mathbb{G}$ (whose distribution is ergodic under lattice shifts) on the undirected graph $(\mathbb{Z}^d, \mathcal{E}^d)$, satisfying conditions similar to the items in our Theorem \ref{thm: construction}. The main changes to these conditions are a) there are no cycles (i.e.~``miniloops'' are disallowed), b) $C_x$ is allowed to be infinite, and c) paths are assumed to pass hyperplanes: given any infinite directed path $(x_0, x_1, \ldots)$ in $\mathbb{G}$, we have for each $k$,  $x_i \cdot e_1 > k$ for all large $i$. (The model of \cite{CK} is defined more generally, but these are the assumed conditions for their theorems about component structure on $\mathbb{Z}^2$.)

    In this setting, in the case $d = 2$, Chaika-Krishnan show a dichotomy: either $C_x$ is a.s.~finite for each $x$ and also the undirected version of $\mathbb{G}$ has one component, or a.s.~each infinite directed path in $\mathbb{G}$ contains a site $x$ with $\#C_x = \infty$ (and the undirected version of $\mathbb{G}$ must have infinitely many components). In their language, one says that $\mathbb{G}$ either exhibits coalescence without bi-infinite trajectories, or each component contains a bi-infinite trajectory. While assumption c) is natural in some coalescing walk models (notably first-passage percolation), there are many examples for which it fails, and for these, the dichotomy can be false. See the example in Section~\ref{sec: zerner} below, which exhibits neither bi-infinite trajectories nor coalescence, its undirected version having two infinite components a.s.

    Our Theorem \ref{thm: two_components} shows that this example demonstrates the most extreme failure of coalescence allowed: if $\mathbb{G}$ does not exhibit bi-infinite trajectories, its undirected version has at most two components a.s. In other words, if the undirected version of $\mathbb{G}$ has at least three infinite components, then $\mathbb{G}$ must exhibit bi-infinite trajectories. It is perhaps worth noting that the dichotomy breaks down in other ways without assumption c): for instance, see \cite{BD} for a model which exhibits a bi-infinite trajectory and also exhibits coalescence. Perhaps the techniques used to prove Theorem \ref{thm: two_components} can be used to completely classify the allowed behavior of stationary coalescing walks which do not necessarily pass hyperplanes.


\section{Proofs}
In this section, we prove the main results, starting with Theorem~\ref{theorem: structure of path in N} in Section~\ref{sec: second_proof}, moving to Theorem~\ref{thm: two_components} in Section~\ref{sec: dimension_two_proof}, and finishing with Theorem~\ref{thm: construction} in Section~\ref{sec: first_proof} and Corollary~\ref{cor: main_result} in Section~\ref{sec: corollary}.

\subsection{Proof of Theorem~\ref{theorem: structure of path in N}}\label{sec: second_proof}

Throughout the proof we make assumption {\bf A}. We first note that for $x,y,z \in \mathbb{Z}^d$ with $x \neq z$, from {\bf A2}, 
\begin{equation}\label{eq: monotonicity}
\text{a.s. if } \langle x,y\rangle \text{ and } \langle y,z \rangle \text{ are edges in } \mathcal{N}_D \text{ then } \omega(\{x,y\}) > \omega(\{y,z\}).
\end{equation}
As a consequence, we have the following facts:
\begin{enumerate}
\item $\mathcal{N}_D$ a.s.~has no directed cycles of length at least $3$. To show this, consider an outcome $\omega$ for which $\omega(e) \neq \omega(f)$ for all $e \neq f$ and a directed cycle with vertices $x_0, x_1, \dots, x_k$ such that $\langle x_i, x_{i+1}\rangle \in \mathcal{N}_D$ for $i=0, \dots, k-1$, with $x_0, \dots, x_{k-1}$ all distinct and $x_0=x_k$ (so that the cycle has length $k$). Writing $x_{k+1}$ also for $x_1$, note that if $k \geq 3$ then for each $i = 0, \dots, k-1$, the edges $\{x_i,x_{i+1}\}$ and $\{x_{i+1},x_{i+2}\}$ are distinct and share an endpoint $x_{i+1}$, so $\omega(\{x_i,x_{i+1}\}) > \omega(\{x_{i+1},x_{i+2}\})$. Iterating this bound, we obtain 
\[
\omega(\{x_k,x_{k+1}\}) = \omega(\{x_0,x_1\}) > \dots > \omega(\{x_{k-1},x_k\}) > \omega(\{x_k,x_{k+1}\}),
\]
a contradiction.
\item For $x \in \mathbb{Z}^d$, write $\Gamma_x$ for the subgraph of $\mathcal{N}_D$ induced by the vertices $y$ such that $x \to y$ in $\mathcal{N}_D$. Then a.s.~there are two possibilities:
\begin{enumerate}
\item $\Gamma_x$ is finite and so is $C_x$. Furthermore $\Gamma_x$ ends in a cycle of length two.
\item $\Gamma_x$ is an infinite vertex self-avoiding directed path.
\end{enumerate}
To see why, since each vertex has out-degree one, we can follow each out-edge starting at $x$ and label the vertices in order as $x=x_0, x_1, \dots$. There are two possibilities: either all $x_i$'s are distinct, or there is a first $i \geq 2$ such that $x_i$ is an element of $\{x_0, \dots, x_{i-1}\}$. In the first case, $\Gamma_x$ is an infinite vertex self-avoiding directed path. In the second, item 1 implies that $x_i = x_{i-2}$, and so $\Gamma_x$ ends in a cycle of length two. We are left to prove that in this case, $C_x$ is finite. 

Suppose for a contradiction that, 
\begin{equation}\label{eq: finite_gamma_x_assumption}
\mathbb{P}(C_0 \text{ is infinite but } \Gamma_0 \text{ is finite}) > 0.
\end{equation}
Then, as above, on the event in \eqref{eq: finite_gamma_x_assumption} we can follow $\Gamma_0$ forward until we reach a cycle of length two. Define the following random variable (``mass transport function'') $m(x,y)$ for $x,y \in \mathbb{Z}^d$:
\[
m(x,y) = \begin{cases}
1 &\quad \text{if } y \text{ is in the two-cycle at the end of } \Gamma_x \\
0 &\quad \text{otherwise}.
\end{cases}
\]
(If $\Gamma_x$ is infinite, then $m(x,y) = 0$ for all $y$.) Because $m$ satisfies $m(x,y)(\omega) = m(x-z,y-z)(T^z\omega)$ for $x,y,z \in \mathbb{Z}^d$, where $T^z$ was defined in \eqref{eq: T_z}, {\bf A1} implies that we can apply the mass transport principle (see \cite{haggstrom, lyonsperes} for an introduction) to obtain
\begin{equation}\label{eq: mass_transport}
\mathbb{E} \sum_{x \in \mathbb{Z}^d} m(x,0) = \mathbb{E} \sum_{x \in \mathbb{Z}^d} m(0,x).
\end{equation}
Since $\sum_x m(0,x) \leq 2$ a.s., the left side is $\leq 2$. However if $C_0$ is infinite but $\Gamma_0$ is finite, the vertices $y,z$ in the two-cycle at the end of $\Gamma_0$ satisfy $m(w,y) = m(w,z) = 1$ a.s.~for all $w \in C_0$. Therefore by \eqref{eq: finite_gamma_x_assumption}, 
\[
\mathbb{P}\left( \sum_{w \in \mathbb{Z}^d} m(w,y) = \infty \text{ for some } y \in \mathbb{Z}^d\right) > 0.
\]
By {\bf A1}, this implies that the left side of \eqref{eq: mass_transport} is infinity, a contradiction. We conclude that a.s., if $\Gamma_x$ is finite for some $x$, then $C_x$ is also finite.

\item For $x \in \mathbb{Z}^d$, a.s. if $\Gamma_x$ is infinite, and we write its vertices in order as $x=x_0, x_1, \dots$, then
\[
\omega(\{x_i,x_{i+1}\}) > \omega(\{x_{i+1},x_{i+2}\}) \text{ for all } i \geq 0.
\]
In other words, $\Gamma_x$ is a monotone decreasing path. This follows from \eqref{eq: monotonicity} and item 2: $\Gamma_x$ is an infinite vertex self-avoiding directed path and we can apply \eqref{eq: monotonicity} to each pair of adjacent edges.
\end{enumerate}

Given the three properties above, we continue with the proof of Theorem~\ref{theorem: structure of path in N}. For a contradiction, we assume that 
\begin{equation}\label{eq: for_a_contradiction}
\mathbb{P}(C_0 \text{ is infinite})>0.
\end{equation}
By {\bf A1}, if we show that \eqref{eq: for_a_contradiction} is false, then Theorem~\ref{theorem: structure of path in N} will follow. By item 2 above, we find
\[
\mathbb{P}(C_0 \text{ and } \Gamma_0 \text{ are infinite})>0.
\]
By item 3, if we define for $x \in \mathbb{Z}^d$
\begin{align*}
I_x&= \inf\{\omega(\{u,v\}) : \langle u,v\rangle \text{ is an edge of } \Gamma_x\} \text{ and } \\
S_x&= \sup\{\omega(\{u,v\}) : \langle u,v \rangle \text{ is an edge of } \Gamma_x\},
\end{align*}
then $\mathbb{P}(C_0 \text{ and } \Gamma_0 \text{ are infinite, and } I_0 < S_0) > 0$, and we can therefore find (a deterministic) $r$ such that
\begin{equation}\label{eq: r_def}
\mathbb{P}(C_0 \text{ and } \Gamma_0 \text{ are infinite, and } I_0 < r <  S_0) > 0.
\end{equation}

Following \eqref{eq: r_def}, for $x,y \in \mathbb{Z}^d$, we say that $y$ is the ``$r$-descendant of $x$'' if
\begin{enumerate}
\item $y$ is a vertex of $\Gamma_x$ and, writing $z$ for the a.s.~unique vertex such that $\langle y,z \rangle$ is an edge of $\mathcal{N}_D$, one has $\omega(\{y,z\}) \geq r$, and
\item for any edge $\langle u,v \rangle$ of $\Gamma_z$, one has $\omega(\{u,v\}) < r$.
\end{enumerate}
In other words, $y$ is the last vertex in $\Gamma_x$ whose out-edge has weight $\geq r$. We will invoke the mass transport principle \eqref{eq: mass_transport} using the transport
\[
m(x,y) = \begin{cases}
1 & \quad \text{if } y \text{ is the }r\text{-descendant of } x \\
0 &\quad\text{otherwise},
\end{cases}
\]
noting again that $m(x,y)(\omega) = m(x-z,y-z)(T^z\omega)$ for  $x,y,z \in \mathbb{Z}^d$. Because $\sum_{x \in \mathbb{Z}^d} m(0,x) \leq 1$ a.s., we obtain 
\begin{equation}\label{eq: pizza_head}
\mathbb{E}\sum_x m(x,0) = \mathbb{E}\sum_{x \in \mathbb{Z}^d} m(0,x) \leq 1.
\end{equation}
However a.s. on the event in \eqref{eq: r_def}, there is a vertex $y$ such that $\sum_x m(x,y) = \infty$. Indeed, if $I_0 < r < S_0$ and $\Gamma_0$ is infinite, item 3 gives that the weights along $\Gamma_0$ are decreasing, so $0$ has an $r$-descendant. Furthermore, for each vertex $w$ of $C_0$ (of which there are infinitely many) a.s.~the graph $\Gamma_w$ is also a vertex self-avoiding infinite directed path equal to $\Gamma_0$ with finitely many directed edges appended in sequence to the beginning. All of these edges by item 3 have weight $> r$ and so any such $w$ has the same $r$-descendant as does $0$. Writing $y$ for this $r$-descendant, we obtain
\[
\text{for an outcome as above, } \sum_x m(x,y) \geq \sum_{w \text{ a vertex of } C_0} m(w,y) = \infty.
\]
Therefore from \eqref{eq: r_def} we deduce that
\[
\mathbb{P}\left( \text{for some }y \in \mathbb{Z}^d,~\sum_x m(x,y) = \infty\right) > 0.
\]
By {\bf A1}, we obtain $\mathbb{E}\sum_x m(x,0) = \infty$ and this contradicts \eqref{eq: pizza_head}. We find then that \eqref{eq: r_def} must have been false, and therefore so was \eqref{eq: for_a_contradiction}. This completes the proof of Theorem~\ref{theorem: structure of path in N}.

\subsection{Proof of Theorem~\ref{thm: two_components}}\label{sec: dimension_two_proof}
The proof will be split over two subsections. In Section~\ref{sec: topology}, we derive some basic results about vertex sets and their boundaries. In Section~\ref{sec: topology_2}, we analyze the structure of infinite components in $\mathcal{N}$ and give the proof of Theorem~\ref{thm: two_components}.

To do this, we begin with some simple definitions. If $x,y \in \mathbb{Z}^2$, then $x$ and $y$ are site-neighbors if $\|x-y\|_1 = 1$ (this is just a redefinition of ``nearest-neighbors'' in $\mathbb{Z}^2$, made to distinguish from neighbors in $\mathcal{N}$). A set $V \subset \mathbb{Z}^2$ of vertices is site-connected if for each $x,y \in V$, there is a path (here a sequence of vertices $x = x_0, x_1, \dots, x_n = y$ such that $x_i$ and $x_{i+1}$ are site-neighbors for all $i$) from $x$ to $y$ remaining in $V$. A site-component of $V$ is a maximal site-connected subset of $V$. As usual, in addition to the (primal) lattice $\left(\mathbb{Z}^2, \mathcal{E}^2\right)$ with edge set $\mathcal{E}^2$ consisting of those edges between neighboring vertices, we use the dual lattice $\left( \left( \mathbb{Z}^2\right)^*, \left( \mathcal{E}^2\right)^*\right)$, with vertex set
\[
\left( \mathbb{Z}^2\right)^* = \mathbb{Z}^2 + \left( \frac{1}{2}, \frac{1}{2}\right)
\]
and edge set $\left( \mathcal{E}^2\right)^*$ consisting of edges between neighbors. Each dual edge $e^*$ bisects a unique edge $e$.

\subsubsection{Basic topological properties of vertex sets}\label{sec: topology}
In this section, we derive some simple properties of vertex sets. These will be used in the main proof in the following section.

\begin{definition}
Let $V \subset \mathbb{Z}^2$.
\begin{enumerate}
\item The closure of $V$, written $\overline{V}$, is the union of $V$ with all finite site-components of $V^c$.
\item The dual edge boundary of $V$, written $B(V)$, is the subgraph of $\left( \left( \mathbb{Z}^2\right)^*, \left( \mathcal{E}^2\right)^* \right)$ induced by the set of dual edges whose unique bisecting edge $\{x,y\}$ has $x \in \overline{V}$ and $y \notin \overline{V}$.
\end{enumerate}
\end{definition}

We note some simple properties of the definitions. The set $\overline{V}$ is intended to be $V$ ``with its holes filled in,'' so its complement should only have infinite site-components:
\begin{lemma}\label{lem: infinite_complement_components}
Let $V \subset \mathbb{Z}^2$. Then $\left( \overline{V}\right)^c$ has only infinite site-components.
\end{lemma}
\begin{proof}
Assume for a contradiction that $\left(\overline{V}\right)^c$ had a finite site-component containing a vertex $x$. Since $\left(\overline{V}\right)^c \subset V^c$ and the finite site-components of $V^c$ are in $\overline{V}$, $x$ must be in an infinite site-component of $V^c$. Then pick a vertex self-avoiding path $\pi$, starting from $x$ and remaining in $V^c$, which has infinitely many vertices. This path $\pi$ must leave the site-component of $\left(\overline{V}\right)^c$ containing $x$, so it contains a vertex $q \in \overline{V}$. But $q \in \pi \subset V^c$, so $q$ must be in $\overline{V} \setminus V$ and therefore is in a finite site-component $\tilde{C}_q$ of $V^c$. However then $\pi \subset \tilde{C}_q$, giving a contradiction since $\pi$ is infinite.
\end{proof}

Next we show that the vertices on the site-boundary of $\overline{V}$ are actually in $V$. (Otherwise, they would be in finite holes in the complement of $V$.) As a consequence, if $V$ is site-connected, so is $\overline{V}$.
\begin{lemma}\label{lem: neighbor_hole}
If $x \in \overline{V}$ has a neighbor in $\left(\overline{V}\right)^c$, then $x \in V$.
\end{lemma}
\begin{proof}
The neighbor $y$ of $x$ that is in $\left(\overline{V}\right)^c$ is not in $V$, and is therefore in an infinite site-component $\tilde{C}$ of $V^c$ (otherwise it would be in $\overline{V}$). If $x$ were not in $V$, then it would be in $\tilde{C}$, as it is adjacent to an element of this site-component. But this means $x$ would be in an infinite site-component of $V^c$ and therefore would not be in $\overline{V}$, a contradiction.
\end{proof}

\begin{lemma}\label{lem: degree_two}
Let $V \subset \mathbb{Z}^2$ be site-connected. Then each vertex in the graph $B(V)$ has degree two. Therefore $B(V)$ is a vertex-disjoint union of vertex self-avoiding circuits and vertex self-avoiding doubly-infinite paths. If $V$ is also infinite, then $B(V)$ contains no vertex self-avoiding circuits.
\end{lemma}
\begin{proof}
Because $B(V)$ is induced by a set of edges, each vertex has at least degree one. Because edges of $B(V)$ separate $\overline{V}$ from $\left( \overline{V}\right)^c$, its vertices must have degree 2 or 4, so we rule out degree 4. Assume for a contradiction that some dual vertex $x$ has degree 4 in $B(V)$. Then we can enumerate the (primal) vertices at distance $\sqrt{2}/2$ of $x$ by $x_1, \dots, x_4$ in clockwise order so that $x_i \in \overline{V}$ if and only if $i$ is odd. Because $x_1$ and $x_3$ are site-neighbors of $\left(\overline{V}\right)^c$, Lemma~\ref{lem: neighbor_hole} implies that they are in $V$. We can then choose a vertex self-avoiding path $\pi$ from $x_1$ to $x_3$ which remains in $V$, and then a plane curve $P$ which starts at $x_1$ and proceeds as follows. First, $P$ connects the vertices of $\pi$ in order by straight line segments. At $x_3$, $P$ connects to $x_1$ by a diagonal line segment (going through $x$). $P$ is a Jordan curve, and therefore its complement (in $\mathbb{R}^2$) has two components: one bounded (its interior) and one unbounded (its exterior). Proceeding in a straight line from $x_2$ to $x_4$, we cross $P$ exactly once. Since neither $x_2$ nor $x_4$ is on $P$, one must be in each component. By symmetry, let's say that $x_2$ is in the interior. Then $x_2 \in V^c$ and must be in a bounded site-component of $V^c$, since any infinite vertex self-avoiding path starting at $x_2$ must leave the interior of $P$ and therefore touch $\pi$ (it cannot touch the interior of the other segment composing $P$). This is a contradiction, since $x_2 \in \left( \overline{V}\right)^c$, and $\overline{V}$ contains the bounded site-components of $V^c$. We conclude that $x$ must have degree 2 in $B(V)$.

Now suppose that $V$ is infinite and, for a contradiction, assume that $B(V)$ contains a vertex self-avoiding circuit. Again, form a Jordan curve $P$ by proceeding along the circuit, using straight line segments to connect its vertices. Let $x \in V$ and $y \in \left( \overline{V}\right)^c$ be such that the edge $\{x,y\}$ bisects a dual edge in $P$. Then $x$ and $y$ are in different components of the complement of $P$ (the edge $\{x,y\}$ crosses $P$ exactly once), so one of them is in the bounded component. First suppose that it is $x$; then because $V$ is infinite and connected, there is an infinite vertex self-avoiding path $\pi_x$ started at $x$ which remains in $V$. But $\pi_x$ must then exit the bounded component of the complement of $P$ and cross $P$ to a vertex of $\left( \overline{V}\right)^c$, which is a contradiction. If instead $y$ is in the bounded component, then by Lemma~\ref{lem: infinite_complement_components}, we can choose an infinite vertex self-avoiding path $\pi_y$ started at $y$ which remains in $\left( \overline{V}\right)^c$. By the same reasoning, we obtain another contradiction. Therefore $B(V)$ contains no vertex self-avoiding circuit. 
\end{proof}

\subsubsection{Topology of nearest neighbor components}\label{sec: topology_2}

%
%
%

Our aim is to show that for $d=2$, if the number of infinite components of $\mathcal{N}$ is $\geq 2$, then it must be 2. Essentially we will show that in this case, the two components are both topologically half-planes which can come within distance 1 of each other, and they may be separated by finite components of $\mathcal{N}$.  The first step is to show that if $C$ is an infinite component of $\mathcal{N}$, then the closure of its vertex set is topologically either a full-plane or half-plane. We do this in the following proposition.
\begin{proposition}\label{prop: double}
For any component $C$ of $\mathcal{N}$, write $V(C)$ for its vertex set and $B(C)$ for $B(V(C))$. Then under assumption {\bf A}, a.s., for each infinite component $C$ of $\mathcal{N}$, $B(C)$ is either empty or is a vertex self-avoiding doubly-infinite path.
\end{proposition}
\begin{proof}
Assume for a contradiction that
\begin{equation}\label{eq: for_a_contradiction_double}
\mathbb{P}\left( \begin{array}{c}
B(C) \text{ contains at least two distinct vertex self-avoiding} \\
\text{ doubly-infinite paths for some infinite component }C \text{ of } \mathcal{N}
\end{array}
\right) > 0.
\end{equation}
We will show that for almost every outcome in this event, there is a vertex $x$ with $\#C_x = \infty$. This will be a contradiction, as we have shown in Theorem~\ref{theorem: structure of path in N} that this has zero probability.

By Lemma~\ref{lem: degree_two}, the two doubly-infinite paths in \eqref{eq: for_a_contradiction_double} can be assumed to be vertex disjoint. So, for such an outcome, let $C$ be an infinite component of $\mathcal{N}$ and let $\pi_1$ and $\pi_2$ be vertex disjoint, vertex self-avoiding doubly-infinite paths contained in $B(C)$. Choose a dual edge $e_0$ in $\pi_1$ and enumerate the edges in either direction along $\pi_1$ as $\dots, e_{-1},e_0,e_1, \dots$. For $n \in \mathbb{Z}$, write $\{x_n,x_n'\}$ for the edge which bisects $e_n$, so that $x_n \in V(C)$ and $x_n' \in \left( \overline{V(C)}\right)^c$. For $n \geq 1$, since $x_n,x_{-n}$ are in $V(C)$, there is a vertex self-avoiding path $P_n$ in $\mathcal{N}$ connecting $x_n$ to $x_{-n}$. Fixing any $y \in V(C)$ such that some edge $\{y,y'\}$ bisects an edge of $\pi_2$, we can also find a path $\pi_3$ in $\mathcal{N}$ connecting $x_0$ to $y$. See Figure~\ref{fig: two_components_1}.

\begin{figure}
\hbox{\hspace{4cm}\includegraphics[width=.45\textwidth, trim={5cm 8cm 16cm 4cm}, clip]{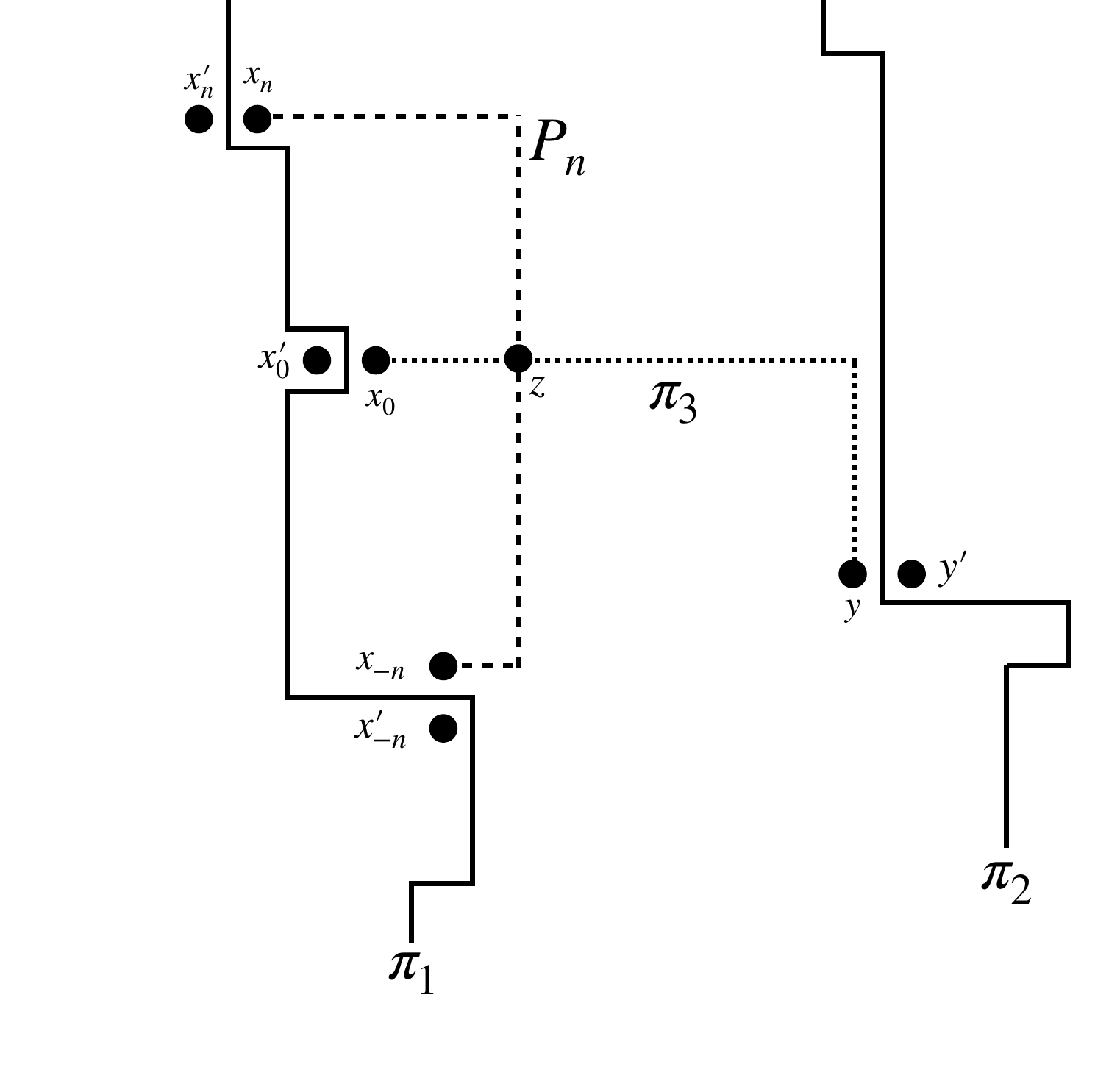}}
  \caption{Illustration of the argument for Proposition~\ref{prop: double}. The closure $\overline{V(C)}$ of the component $C$ lies between the two paths $\pi_1,\pi_2$, which are part of the boundary $B(C)$. The $x_i$'s are the endpoints of edges dual to those of $\pi_1$ that lie in $V(C)$, and $P_n$ is a path in $\mathcal{N}$ connecting $x_n$ to $x_{-n}$. The path $\pi_3$ is in $\mathcal{N}$ and connects $x_0$ to $y$, an endpoint of an edge dual to one in $\pi_2$. By the Jordan curve theorem, $P_n$ must intersect $\pi_3$ at some vertex $z$. Because this is true for all $n$, one can argue that some such $z$ has $\#C_z =\infty$.}
  \label{fig: two_components_1}
\end{figure}

We will now argue by the Jordan curve theorem that $P_n$ must intersect $\pi_3$. To do this, define the plane curve $\mathcal{P}_n$ as follows. It proceeds from $x_n$ to $x_{-n}$ along the edges of $\mathcal{N}$ that connect vertices of $P_n$, then it connects $x_{-n}$ to $\pi_1$ halfway through the edge $\{x_{-n},x'_{-n}\}$, then it proceeds along $\pi_1$ until it meets the edge $\{x_n,x'_n\}$, and last moves to $x_n$ halfway through this edge. The curve $\mathcal{P}_n$ is a Jordan curve, and so its complement in $\mathbb{R}^2$ has a bounded component (the interior) and an unbounded component (the exterior). We first note that 
\begin{equation}\label{eq: exterior_P_n}
\text{ for } n \geq 1, \left( \overline{V(C)}\right)^c \subset \text{ext}~\mathcal{P}_n.
\end{equation}
Indeed, if $u \in \left( \overline{V(C)}\right)^c$, then (by Lemma~\ref{lem: infinite_complement_components}) $u$ is in an unbounded site-component of $\left( \overline{V(C)}\right)^c$, and so there is an infinite vertex self-avoiding path $\pi$ starting from $u$ and remaining in this component. However $\pi$ cannot touch $\mathcal{P}_n$ because it never leaves this component and therefore never comes in contact with $P_n$ or $\pi_1$. Therefore $u$ must be in the unbounded component of the complement of $\mathcal{P}_n$, and this shows \eqref{eq: exterior_P_n}.

Due to \eqref{eq: exterior_P_n}, we can now argue that $x_0$ is either on $\mathcal{P}_n$ or in its interior. So suppose that $x_0$ is not on $\mathcal{P}_n$; this means it is not on $P_n$. Then the edge $\{x_0,x_0'\}$ only touches $\mathcal{P}_n$ at $\pi_1$, and it only touches it once. However $x_0'$ is in the exterior of $\mathcal{P}_n$, so $x_0$ must be in the interior. By similar reasoning, we can argue that $y$ (the other endpoint of $\pi_3$) is either on $\mathcal{P}_n$ or in its exterior. Indeed, if it is not on $\mathcal{P}_n$, then it is not on $P_n$, and then the edge $\{y,y'\}$ does not touch $\mathcal{P}_n$ but ends at the vertex $y'$, which is in the exterior. Therefore $y$ is also in the exterior. Because $x_0$ is not in the exterior of $\mathcal{P}_n$, and $y$ is not in the interior, the curve formed by following straight line segments between the vertices of $\pi_3$ must touch $\mathcal{P}_n$. Since $\pi_3$ doesn't leave $C$, it must touch a vertex of $P_n$. Therefore we find that for all $n$, $P_n$ shares a vertex with $\pi_3$, as desired. 

Because $\pi_3$ is finite, there is a $z \in \pi_3$ that is in $P_{n_k}$ for some subsequence $(n_k)$ of integers with $n_k \to \infty$. We claim that either $x_{n_k} \to z$ in $\mathcal{N}_D$ or $x_{-n_k} \to z$ in $\mathcal{N}_D$. This is clear if one of $x_{n_k}$ or $x_{-n_k}$ equals $z$. Otherwise, $z$ has degree two on $P_{n_k}$. Since a.s., all vertices have out-degree one in $\mathcal{N}_D$ (and therefore so does $z$), one of its neighbors on $P_{n_k}$, say $w_0$, is such that $\langle w_0,z\rangle$ is an edge of $\mathcal{N}_D$. Enumerating the subsequent vertices (beyond $w_0$) of $P_{n_k}$ as $w_1, \dots, w_r$, each $w_i$ has out-degree one in $\mathcal{N}_D$, so $\langle w_{i+1}, w_i\rangle$ is an edge of $\mathcal{N}_D$. Thus $w_r \to z$ in $\mathcal{N}_D$. Since $w_r = x_{n_k}$ or $x_{-n_k}$, this proves the claim and furthermore establishes that \eqref{eq: for_a_contradiction_double} implies
\[
\mathbb{P}\left( \exists z \in \mathbb{Z}^2 \text{ such that } \#C_z = \infty\right) > 0.
\]
This is a contradiction and completes the proof.
\end{proof}

Now that each infinite component of $\mathcal{N}$ is topologically a half-space, we must show that the complement cannot contain more than one other infinite component. This could happen, for example, if three infinite components were separated by an infinite union of finite components with at least three topological ends, or if three infinite components come within distance 1 of each other. To rule out these and other possibilities, we will make use of the results of Burton-Keane in \cite{BK2}. To begin, we make a few definitions, and state a structural lemma which follows from arguments of \cite{BK2}.
\begin{definition}
Write $X$ for the subset of $\mathbb{Z}^2$ defined by
\[
X = \mathbb{Z}^2 \setminus \bigcup_{C} \overline{V(C)},
\]
where the union is over all infinite components of $\mathcal{N}$. Then $\mathbb{Z}^2$ is the disjoint union of sets of the following three types:
\begin{enumerate}
\item[(a)] $\overline{V(C)}$ for an infinite component $C$ of $\mathcal{N}$,
\item[(b)] an infinite site-component of $X$, and
\item[(c)] a finite site-component of $X$.
\end{enumerate}
\end{definition}
Note that because the complement of $X$ contains only infinite site-components, each set of type (b) or (c) is equal to its closure. For the next lemma, we say that vertices $x$ and $y$ are $\ast$-neighbors (or $\ast$-adjacent) if $\|x-y\|_\infty = 1$. We extend this notion to sets in the usual way.

\begin{lemma}\label{lem: BK2}
Under assumption {\bf A}, a.s., none of the following occur.
\begin{enumerate}
\item There is a set of type (b) whose complement has at least 3 site-components.
\item There is a dual vertex within Euclidean distance $\sqrt{2}/2$ of three different sets of types (a) or (b).
\item There is a set of type (c) which is $\ast$-adjacent to at least three different sets of types (a) or (b).
\end{enumerate}
\end{lemma}
 \begin{proof}
Item 1 is a direct application of \cite[Thm.~2]{BK2}, which states that in stationary site percolation on $\mathbb{Z}^2$, a.s., there is no ``ribbon'' whose complement contains at least 3 site-components. To apply this result, we define variables $(x_v)_{v \in \mathbb{Z}^2}$ as $x_v = \mathbf{1}_{\{v \in X\}}$. The $x_v$ form a stationary site percolation, and any infinite site-component of $X$ (a type (b) set) is a ``1-ribbon'' (the closure of an infinite $1$-cluster). The result follows.

Items 2 and 3 have similar proofs, and follow that of \cite[Thm.~1]{BK2} (see also \cite{G}) , which states that in stationary site percolation on $\mathbb{Z}^2$, a.s., no ``rock'' has at least 3 ribbons as $\ast$-neighbors. Because the details are the same, we only sketch the proofs. To any dual vertex $v$ as described in item 2, we associate three infinite vertex self-avoiding site-paths $\gamma_i(v)$ in different sets of types (a) or (b), starting from vertices within Euclidean distance $\sqrt{2}/2$ of $v$. For any such $v \in [0,n]^2$, consider the first intersection $x_i(v)$ of $\gamma_i(v)$ with $\partial [0,n]^2$, and ordering them so that $x_1(v),x_2(v),x_3(v)$ are in counterclockwise order, we say that $x_2(v)$ is the ``central point'' associated with $v$. (It appears that even in the original argument of \cite{BK2}, a bit more care needs to be taken to define the central point: the $x_i(v)$ should be chosen as functions of the intersection of their corresponding (a) or (b) set with $[0,n]^2$.) One then uses a Jordan curve argument to prove that central points corresponding to different $v$'s are distinct, and therefore there can be at most $4n$ such $v$'s in $[0,n]^2$. (In this part, it is important that for given $v,v'$, the starting points of the $\gamma_i(v')$'s must all lie in the closure of one of the regions between the $\gamma_i(v)$'s.) However, if the event described in item 2 has positive probability, translation invariance implies that the expected number of $v$ in $[0,n]^2$ is of order $n^2$, a contradiction for large $n$.

The argument for item 3 is similar, defining three paths corresponding to each type (c) set as described, paths $\gamma_i$, and corresponding central points. Again, central points associated to distinct such type (c) sets are distinct, and we conclude as above.
 \end{proof}
 
 Given these preparations, we now prove that $\mathcal{N}$ has at most 2 infinite components.
\begin{proof}[Proof of Theorem~\ref{thm: two_components}]
We start by defining the $\ast$-boundary for an infinite component $C$ of $\mathcal{N}$, using our dual path $B(C)$. Writing $\mathsf{D}$ for the event in Proposition~\ref{prop: double}, consider a configuration in $\mathsf{D}$ in which $\mathcal{N}$ has at least 2 infinite components. Let $C$ be any one of them, and note that $B(C)$ is a (nonempty) doubly-infinite, vertex self-avoiding dual path. Enumerate the dual edges of $B(C)$ as $\dots, e_{-1}, e_0, e_1, \dots$, and write $x_i$ for the vertex of $\left( \overline{V(C)}\right)^c$ which is an endpoint of the edge whose dual is $e_i$. The sequence $(x_n)$ is $\ast$-connected (it is in fact a $\ast$-path), but it not necessarily site-connected. To remedy this, we define a doubly-infinite sequence $P = P(C)$ of vertices by following the $x_i$'s, but inserting between any $x_i$ and $x_{i+1}$ which are $\ast$-neighbors but are not site-neighbors their unique common site-neighbor which is in $\left(\overline{V(C)}\right)^c$. (For example, if $x_i = (0,0)$ and $x_{i+1} = (1,1)$, with $e_i = \{(-1/2,1/2),(1/2,1/2)\}$ and $e_{i+1} = \{(1/2,1/2),(1/2,3/2)\}$, then this neighbor is the vertex $(1,0)$. In fact, this is the only possible case up to translation, reflection, and rotation by multiples of $\pi/2$.) Then $P$ is a natural enumeration of the $\ast$-boundary of $\overline{V(C)}$ and is clearly site-connected (it is a path which might not be vertex self-avoiding). For our given $C$ and $P$, we will consider the different possible sets of types (a)-(c) which can intersect $P$. 

First, we argue that no infinite component can be a $\ast$-neighbor to two other sets of type (a) or (b):
\begin{equation}\label{eq: triple_touch_assumption}
\mathbb{P}\left( \begin{array}{c}
\exists x,y \in P = P(C) \text{ in different sets of type (a)}\\
\text{or (b) for some infinite component } C \text{ of } \mathcal{N}
\end{array}
\right) = 0.
\end{equation}
For a contradiction, assume that this probability is positive. We will show that in this case, we can find either a dual vertex as in item 2 of Lemma~\ref{lem: BK2} or a set of type (c) as in item 3 of Lemma~\ref{lem: BK2}. To do this, consider an outcome in the intersection of $\mathsf{D}$ with the event in \eqref{eq: triple_touch_assumption}, with $C,P,x,y$ as described, and write $B_1$ for the set of type (a) or (b) containing $x$. Following $P$ from $x$ to $y$, we must exit $B_1$. 

{\it Case 1.} If we enter a different type (a) or (b) set, say $B_2$, then write $u$ for the last point in $P$ of $B_1$ before entering $B_2$, and $v$ for the first point of $B_2$. If $u = x_n$ for some $n$, then after translating, rotating, and reflecting, we may assume that $u = (0,0)$ and $(0,1) \in V(C)$ (so that $e_n = \{(-1/2,1/2),(1/2,1/2)\} \in B(C)$). If $e_{n+1} = \{(1/2,1/2),(3/2,1/2)\}$, then $v = x_{n+1} = (1,0)$ then the dual vertex $(1/2,1/2)$ satisfies item 2 of Lemma~\ref{lem: BK2} with sets $B_1,B_2, \overline{V(C)}$. (See Figure~\ref{fig: two_components_cases}.) If instead $e_{n+1} = \{(1/2,1/2),(1/2,3/2)\}$, then $v = (1,0)$ is in $P$ but not in the sequence $(x_n)$; however, $(1/2,1/2)$ still satisfies item 2 with sets $B_1,B_2,\overline{V(C)}$. We cannot have $e_{n+1} = \{(1/2,1/2),(1/2,-1/2)\}$ because then $v = x_n = u$, but $u$ and $v$ are in different sets of type (a) or (b). The last possibility is that $u$ is not any of the $x_n$'s, in which case after translating, rotating, and reflecting, we may assume that $u = (1,0)$ and $v = (1,1)$ with $(0,1) \in V(C)$. Here, again the vertex $(1/2,1/2)$ satisfies item 2 with the same sets $B_1,B_2,\overline{V(C)}$.

\begin{figure}
\hbox{\hspace{2cm}\includegraphics[width=.8\textwidth, trim={7cm 13cm 7cm 8cm}, clip]{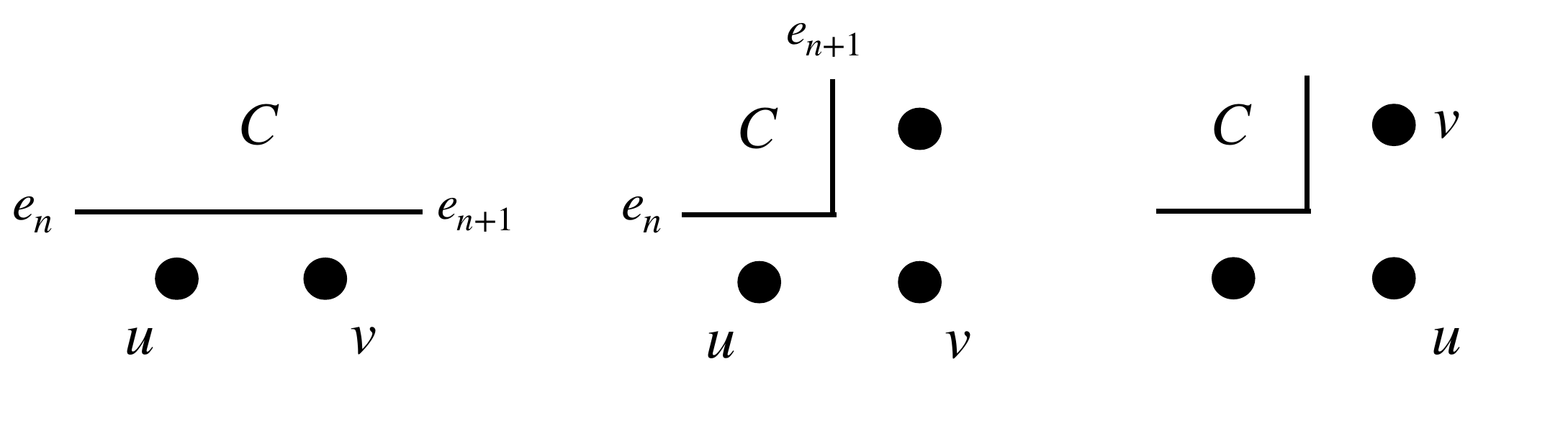}}
  \caption{The three possibilities in the argument of case 1 for equation \eqref{eq: triple_touch_assumption}. In the first, $u=x_n$ for some $n$ and $v$ is as well. In the second, $u= x_n$ for some $n$, but $v$ is not, although it is in $P$. In the last, $u \in P$ is not of the form $x_n$ but $v$ is. In all possibilities, the central dual vertex (taken to be $(1/2,1/2)$) satisfies the condition of item 2 of Lemma~\ref{lem: BK2}.}
  \label{fig: two_components_cases}
\end{figure}

{\it Case 2.} The other possibility is that, after we leave $B_1$, we enter a set $Y$ of type (c). In this case, $B_1$ must be of type (a). If, as we proceed along $P$, we next re-enter $B_1$, we simply wait until we leave $B_1$ once again. Otherwise, once we leave $Y$, we enter another set $B_2$ of type (a) which is not $B_1$ and also not $\overline{V(C)}$. In this case, the set $Y$ is site-adjacent to both $B_1$ and $B_2$ but also $\ast$-adjacent to $\overline{V(C)}$ (as every vertex of $P$ is $\ast$-adjacent to $\overline{V(C)}$). Therefore $Y$ satisfies item 3 of Lemma~\ref{lem: BK2}.

In either case, assuming that \eqref{eq: triple_touch_assumption} fails implies that at least one of the events described in Lemma~\ref{lem: BK2} has positive probability, a contradiction. Therefore \eqref{eq: triple_touch_assumption} holds.

Next, we argue that if any infinite component of $\mathcal{N}$ is a $\ast$-neighbor of another one (by the above, the component can have at most one such $\ast$-neighbor), then there are exactly two infinite components of $\mathcal{N}$:
\begin{equation}\label{eq: case_2_assumption}
\mathbb{P}\left( \begin{array}{c}
\exists x \in P=P(C) \text{ in a set of type (a) for some infinite}\\
\text{component }C \text{ of } \mathcal{N} \text{ and } \mathcal{N} \text{ has at least }3 \text{ infinite components}
\end{array}
\right) = 0.
\end{equation}
As before, we argue by contradiction and assume that this probability is positive. Choose any outcome in the intersection of $\mathsf{D}$ and this event such that the event in \eqref{eq: triple_touch_assumption} does not occur, and pick $C,x$ as described in \eqref{eq: case_2_assumption}. Note that if $\overline{V(C')}$ is the set of type (a) containing $x$, then infinitely many vertices of $P$ (in both directions) are in $\overline{V(C')}$. (This implies that $P$ consists of vertices of $\overline{V(C')}$ separated by finite segments of vertices in (c) components.) Indeed, if this were not true, then in some direction along $P$, all vertices from some point on would be in $X$ (since they could not be in another type (a) set due to the event in \eqref{eq: triple_touch_assumption} not occurring). But these vertices are site-connected, so they would be part of a set of type (b), and this would also lead us back to the event in \eqref{eq: triple_touch_assumption}.

Because the event in \eqref{eq: case_2_assumption} occurs, we can choose, in addition to the components $C$ and $C'$, yet another infinite component $C''$ of $\mathcal{N}$. Pick a site-connected path $\pi$ from $C''$ to $C$ and follow it until its first vertex in $\overline{V(C)} \cup \overline{V(C')}$ (this cannot be the initial vertex of $\pi$). By symmetry, we may assume that this vertex is in $\overline{V(C)}$. Let $w$ be the vertex of $\pi$ directly before it. Since $w \in P$ (as it is site-adjacent to $\overline{V(C)}$), but not in $\overline{V(C')}$, and vertices of $P$ are either in $\overline{V(C')}$ or in type (c) sets, $w$ must be in a set $Y$ of type (c). Note that then $Y$ is site-adjacent to $\overline{V(C)}$, but it is also site-adjacent to $\overline{V(C')}$, as we can follow $P$ through $Y$ directly to $\overline{V(C')}$. Following $\pi$ backward from $w$ until we exit $Y$, we must enter another type (a) set, but this set cannot be $\overline{V(C)}$ or $\overline{V(C')}$. We conclude that $Y$ is site-adjacent to 3 sets of type (a); that is, $Y$ satisfies the condition of item 3 of Lemma~\ref{lem: BK2}. Just as in case 1, we see that our assumption that \eqref{eq: case_2_assumption} fails implies the event described in item 3 of Lemma~\ref{lem: BK2} has positive probability, a contradiction. Therefore \eqref{eq: case_2_assumption} holds.

Last, we deal with the remaining possibility: that all vertices of $P = P(C)$ are in one set of type (b), and that this holds for all infinite components $C$ of $\mathcal{N}$. In this case, we will show that the number of infinite components of $\mathcal{N}$ is at most two: 
\begin{equation}\label{eq: case_3_assumption}
\mathbb{P}\left( \begin{array}{c}
\text{all }x \in P = P(C) \text{ are in a set of type (b) for all infinite}\\ 
\text{components } C \text{ of } \mathcal{N} \text{ and } \mathcal{N} \text{ has at least }3 \text{ infinite components}
\end{array}
\right) = 0.
\end{equation}
For a contradiction, assume this probability is positive, and consider any outcome in the intersection of $\mathsf{D}$ and this event. Pick infinite components $C_1, C_2, C_3$ of $\mathcal{N}$ and write $P_1, P_2, P_3$ for their corresponding $\ast$-boundary paths. Then $P_i$ is contained entirely in some set $Y_i$ of type (b). In this case, we will show that all $Y_i$'s are equal and that their complement has at least 3 site-components, as in item 1 of Lemma~\ref{lem: BK2}.

Choose a site-connected path $\pi$ from $C_i$ to $C_j$ for some $i \neq j$. Let $w$ be the vertex directly before entering $\overline{V(C_j)}$ for the first time and note that since $w$ is site-adjacent to $\overline{V(C_j)}$, it must be in $P_j$, and therefore in $Y_j$. Following $\pi$ backward from $w$ until it last leaves $Y_j$ at some vertex $z$, we see that $z$ must be in $\overline{V(C_i)}$. The reason is that otherwise $z$ must be in another set of type (a), and this set would have $\ast$-boundary path fully contained in $Y_j$ (as it is site-adjacent to $Y_j$), so to reach $C_i$, we would need to re-enter $Y_j$, contradicting the definition of $z$. We find, therefore, that $Y_j$ is site-adjacent to $\overline{V(C_i)}$ as well, and so $Y_i = Y_j$.

Because $Y_1=Y_2=Y_3$, each of $\overline{V(C_1)}, \overline{V(C_2)}, \overline{V(C_3)}$ are $\ast$-adjacent only to $Y_1$, This implies that the sets $\overline{V(C_i)}$ are contained in different site-components of $Y_1^c$ (otherwise we could move from one to the other without touching $Y_1$). Thus $Y_1$ satisfies the condition of item 1 of Lemma~\ref{lem: BK2}. This means that if we assume that \eqref{eq: case_3_assumption} fails, then the event described in item 1 has positive probability, a contradiction. Therefore \eqref{eq: case_3_assumption} holds. This completes the proof.
\end{proof}

\subsection{Proof of Theorem~\ref{thm: construction}}\label{sec: first_proof}

Suppose that $G = (V,E)$ is a graph such that $E$ is countable and $\mathbb{G} = (V,\mathbb{E})$ is a directed graph with the properties stated in the theorem: if $\langle x,y \rangle \in \mathbb{E}$ then $\{x,y\} \in E$, each $x \in V$ has out-degree one in $\mathbb{G}$, $\mathbb{G}$ has no directed cycles of length at least three, and for each vertex $x \in V$, $C_x$ is finite. We define weights $\omega(e)$ for edges $e\in E$ as follows. Let $(U(e))_{e \in E}$ be a collection of i.i.d.~uniform $[0,1]$ random variables and let $\mathsf{E}$ be the set of edges $\{x,y\} \in E$ such that $\langle x,y \rangle \in \mathbb{E}$. If $e \in \mathsf{E}$ with $e = \{x,y\}$ and $\langle x,y \rangle \in \mathbb{E}$, write $V(e)$ for the number of vertices in $C_x$. Note that $V$ is well-defined: if $\langle x,y \rangle$ and $\langle y, x \rangle$ are both in $\mathbb{E}$, then the graphs $C_x$ and $C_y$ are the same, so they have the same number of vertices. Our definition of $\omega(e)$ is
\begin{equation}\label{eq: omega_def}
\omega(e) = \begin{cases}
\frac{1}{V(e) + U(e)} & \quad \text{if } e \in \mathsf{E} \\
1 + U(e) & \quad \text{otherwise}.
\end{cases}
\end{equation}

Note that a.s., the weights $\omega(e)$ are distinct. Therefore to prove Theorem~\ref{thm: construction}, we will show that a.s., the nearest neighbor graph $\mathcal{N}_D$ constructed from the weights is equal to $\mathbb{G}$. First suppose that $\langle x,y \rangle$ is an edge of $\mathbb{G}$; we will prove that it is in $\mathcal{N}_D$. To do this, we will show that for $z \neq y$ with $\{z,x\} \in E$, we have $\omega(\{z,x\}) > \omega(\{x,y\})$. There are two cases. If $\{z,x\} \notin \mathsf{E}$, then a.s.
\[
\omega(\{z,x\}) = 1 + U(\{z,x\}) > 1 > \frac{1}{V(\{x,y\}) + U(\{x,y\})} = \omega(\{x,y\}),
\]
since $V(\{x,y\}) \geq 1$. Alternatively, if $\{z,x\} \in \mathsf{E}$, then because $x$ has out-degree one in $\mathbb{G}$, $\langle z,x \rangle \in \mathbb{E}$ and therefore $V(\{z,x\})$ is the number of vertices in $C_z$. We claim that $C_z$ is strictly contained in $C_x$: each vertex of $C_z$ is in $C_x$, but $x \notin C_z$. Assuming this for the moment, we obtain $V(\{x,z\}) \leq V(\{x,y\}) - 1$ and so a.s.
\[
\omega(\{x,y\}) = \frac{1}{V(\{x,y\}) + U(\{x,y\})} < \frac{1}{V(\{z,x\}) + 1} < \frac{1}{V(\{z,x\}) + U(\{z,x\})} = \omega(\{z,x\}).
\]
This implies that a.s.~if $\langle x,y \rangle \in \mathbb{E}$ then $\langle x,y \rangle$ is an edge of $\mathcal{N}_D$.

To prove the claim, observe that if $w$ is a vertex of $C_z$ then it is clearly a vertex of $C_x$: any directed path from $w$ to $z$ can be extended to $x$ by simply appending the edge $\langle z,x\rangle$ to the end. So we assume for a contradiction that $x$ is a vertex of $C_z$. Then there is a directed path $\pi$ from $x$ to $z$ in $\mathbb{G}$ (which we may assume is vertex self-avoiding). This path cannot contain an edge from $x$ to $z$ since $\langle x,z \rangle \notin \mathbb{E}$, so it must first visit some $u$ which is not equal to $x$ or $z$. But then appending the edge $\langle z,x \rangle$ to the end of $\pi$ produces a directed cycle of length at least three in $\mathbb{G}$, a contradiction. We conclude that $x$ is not a vertex of $C_z$ and therefore the claim holds.

To complete the proof of Theorem~\ref{thm: construction}, suppose that $\langle x,y \rangle \notin \mathbb{E}$ but $\{x,y\} \in E$. Since $x$ has out-degree one in $\mathbb{G}$, there is some $z \neq y$ such that $\langle x,z \rangle \in \mathbb{E}$. As we saw above, this implies a.s.~that $\langle x,z\rangle$ is an edge of $\mathcal{N}_D$. Since $x$ a.s.~has out-degree at most one in $\mathcal{N}_D$ (as the weights are distinct), we find that $\langle x,y\rangle$ cannot be an edge of $\mathcal{N}_D$. This shows that the edges of $\mathbb{G}$ and $\mathcal{N}_D$ are the same, and finishes the proof.

\subsection{Proof of Corollary~\ref{cor: main_result}}\label{sec: corollary}
To prove the corollary, we will build various random directed graphs and use Theorem~\ref{thm: construction} to exhibit them as nearest neighbor graphs for some choices of weights. Because i.i.d., continuously distributed weights produce nearest neighbor graphs with all finite components, it suffices to take $k \geq 1$. Our underlying graph will be $\mathbb{Z}^d$ for some $d \geq 2$ and we will identify any directed graph $\mathbb{G}$ with vertex set $\mathbb{Z}^d$ with a point in the space $\{0,1\}^{\vec{\mathcal{E}}^d}$, where $\vec{\mathcal{E}}^d$ is the set $\{\langle x,y \rangle : x,y \in \mathbb{Z}^d, \|x-y\|_1 = 1\}$ of directed edges of the lattice. (We give this space the usual product Borel sigma-algebra.) Any translation $T^z$ acts on this space just as it did on the edge-weight space $\Omega$: for $\eta \in \{0,1\}^{\vec{\mathcal{E}}^d}$, we set $(T^z\eta)(\langle x,y \rangle) = \eta(\langle x+z,y+z\rangle)$. Last, to a directed graph $\mathbb{G} = \left(\mathbb{Z}^d, \mathbb{E}\right)$ we naturally associate the point $\eta = \eta(\mathbb{G})$:
\[
\eta(\langle x,y \rangle) = \begin{cases}
1 &\quad \text{if } \langle x,y \rangle \in \mathbb{E} \\
0 &\quad \text{otherwise}.
\end{cases}
\]
By Theorem~\ref{thm: construction}, to show that there is a measure $\mathbb{P}$ satisfying {\bf A} such that a.s. the graph $\mathcal{N}_D$ has, say, property $P$, it suffices to show that there is a random directed graph $\mathbb{G}$ (measure on $\{0,1\}^{\vec{\mathcal{E}}^d}$) which is invariant under all translations $T^z$ such that a.s., $\mathbb{G}$ has property $P$ and the properties stated in the theorem: if $\langle x,y \rangle \in \mathbb{E}$ then $\{x,y\} \in \mathcal{E}^d$, each $x \in \mathbb{Z}^d$ has out-degree one in $\mathbb{G}$, $\mathbb{G}$ has no directed cycles of length at least three, and for each $x \in \mathbb{Z}^d$, $C_x$ is finite. This is because if the distribution of $\mathbb{G}$ is invariant under translations, then the weights $(\omega(e))$ defined in \eqref{eq: omega_def} will be as well, so {\bf A1} will hold, and the weights are all distinct, so {\bf A2} will hold as well.

Given these preliminaries, we move to the constructions.

\subsubsection{Case $d=2$ and $k=2$}\label{sec: zerner}
To construct the measure $\mathbb{P}$ for the case $d=2$ and $k=2$, we use the graph from \cite[Sec.~3]{ZernerMerkl}. It is a.s.~a union of two directed trees, each built from coalescing random walks: one tree moves up-right and the other moves down-left, both trees being dual to each other. Let $(B_x)_{x \in \mathbb{Z}^2}$ be a family of i.i.d.~Bernoulli$(1/2)$ random variables and define a directed graph $\mathbb{G}_0$ using these variables as follows. The vertex set of $\mathbb{G}_0$ is $\mathbb{Z}^2$ and the edge set is
\begin{align*}
&\bigcup_{x : B_x = 1} \left\{ \langle 2x,2x+e_2\rangle, \langle 2x+e_2,2x+2e_2\rangle, \langle 2x+e_1, 2x+e_1-e_2\rangle, \langle 2x+e_1+e_2, 2x+e_1\rangle\right\} \\
\bigcup~&\bigcup_{x : B_x = 0} \left\{ \langle 2x, 2x+e_1\rangle, \langle 2x+e_1, 2x+2e_1\rangle, \langle 2x+e_2, 2x+e_2-e_1\rangle, \langle 2x+e_1+e_2, 2x+e_2\rangle \right\}.
\end{align*}
Here $e_1$ and $e_2$ are the standard basis vectors of $\mathbb{R}^2$. (See \cite[Figs.~2,3]{ZernerMerkl} for illustrations of the structure of $\mathbb{G}_0$.) Although $\mathbb{G}_0$ is not translation invariant, we can remedy this by letting $U$ be an independent uniform vector on the set $\{0, e_1, e_2, e_1+e_2\}$ and setting $\mathbb{G}$ to be the translation of $\mathbb{G}_0$ by $U$. That is, $\mathbb{G}$ has vertex set $\mathbb{Z}^2$ but edge set
\[
\left\{ \langle x+U,y+U\rangle : \langle x,y \rangle \text{ is an edge of } \mathbb{G}_0\right\}.
\]
In \cite[p.~1730]{ZernerMerkl} it is shown that the distribution of $\mathbb{G}$ is invariant (and even ergodic) under lattice translations.

By construction, a.s.~each vertex has out-degree one in $\mathbb{G}$ and the graph $\Gamma_x$ obtained by starting with a vertex $x$ and following each out-edge is a symmetric random walk that, once it reaches a vertex $z$, steps either (a) up twice or right twice if $z\cdot (1,1)$ is even or (b) down twice or left twice if $z\cdot (1,1)$ is odd. Because all paths of type (a) intersect, as do all paths of type (b), but paths of type (a) do not intersect those of type (b), it follows that the undirected version of $\mathbb{G}$ (the graph with the same vertex set but edge set equal to $\{\{x,y\} : \langle x,y \rangle \text{ is an edge of } \mathbb{G}\}$) has exactly two components a.s. (These are the two directed trees mentioned above.) In \cite[p.~1730]{ZernerMerkl} it is shown that ``for any $x$ the subtree for which $x$ is the root is a.s.~finite.'' In our notation, this means that for each $x \in \mathbb{Z}^2$, the number of vertices in $C_x$ is finite. Because this $\mathbb{G}$ a.s.~satisfies the conditions of Theorem~\ref{thm: construction}, is invariant under translations, and its undirected version has two components, this completes the case $d=2$ and $k=2$.

\begin{remark}\label{rem: topology}
In the language of the proof of Theorem~\ref{thm: two_components}, the above example exhibits the lattice $\mathbb{Z}^2$ as a union of two disjoint type-(a) sets, $V_1$ and $V_2$, corresponding to the up-right tree and the down-left tree. The dual edge boundaries $B(V_1)$ and $B(V_2)$ coincide and, furthermore, each vertex of $\mathbb{Z}^2$ is an endpoint of an edge dual to one on this boundary. One can modify this example to produce a model consisting of two type-(a) sets separated by type-(c) sets as follows. For any $z$ such that all vertices $x$ with $\langle x,z \rangle$ an edge of $\mathbb{G}$ satisfy $\#C_x =1$, we remove the out-edge of $z$ from $\mathbb{G}$ and add a new out-edge from $z$ to any such $x$ (choosing one in a deterministic manner). The resulting directed graph is then seen to be a nearest neighbor graph for weights satisfying assumption {\bf A}, and every $z$ listed above becomes part of a type-(c) set separating the two type-(a) sets.

In fact, one can also construct a nearest neighbor graph $\mathbb{G}$ on $\mathbb{Z}^2$ which splits the lattice into two type-(a) sets separated by a type-(b) set. To do this, we start with a stationary site-percolation model which a.s.~exhibits two infinite 1-clusters (which are topological half-planes) separated by an infinite 0-cluster (which is topologically a strip). This can be done by choosing finite order type $I= \{1,2,3\}$ in \cite[p.~309]{BK2}. Then we place independent uniform spanning trees on the subgraphs of $\mathbb{Z}^2$ induced by the infinite 1-clusters, and independent i.i.d.~uniform $(0,1)$ weights on the edges of the subgraph of $\mathbb{Z}^2$ induced by the infinite 0-cluster. Our final graph $\mathbb{G}$ is the union of the spanning trees along with the standard i.i.d.~nearest neighbor model on the 0-cluster. One can show that since each 1-cluster is topologically a half-plane, the spanning trees are one-ended (and therefore we can orient them toward infinity), and thus form the two infinite components of our nearest neighbor graph $\mathbb{G}$. The graph constructed on the 0-cluster is a union of infinitely many finite components of our graph $\mathbb{G}$ and, since the 0-cluster is site-connected, it forms a type-(b) set.
\end{remark}

\subsubsection{Case $d\geq 3$ and $k=\infty$}\label{sec: infinitely_many}
To prove item 2 in the case $d \geq 3$ and $k=\infty$ we use a layered construction. We produce a translation-invariant random graph $\mathbb{G}$ with vertex set $\mathbb{Z}^d$ which satisfies the conditions of Theorem~\ref{thm: construction} and whose undirected version has infinitely many infinite components. We do this by induction, so suppose there is such a random graph $\mathbb{G}_d$ for a given dimension $d$ with at least two infinite components (by the above argument, we know this is true for $d=2$); we will show one exists in dimension $d+1$ with infinitely many infinite components. Write $\eta_d$ for the (random) point in $\{0,1\}^{\vec{\mathcal{E}}^d}$ corresponding to $\mathbb{G}_d$ in $d$-dimensions and define an element $\eta \in \{0,1\}^{\vec{\mathcal{E}}^{d+1}}$ by
\[
\eta\left( \langle x,y \rangle \right) = \begin{cases}
\eta_d\left( \langle \pi_d(x), \pi_d(y) \rangle \right) & \quad \text{if } \pi_d(x) \neq \pi_d(y) \\
0 & \quad \text{otherwise}.
\end{cases}
\]
Here $\pi_d : \mathbb{Z}^{d+1} \to \mathbb{Z}^d$ is the projection $\pi_d(x) = \sum_{i=1}^d (x \cdot e_i)e_i$. If $\mathbb{G}$ is the graph corresponding to $\eta$, then the intersection of $\mathbb{G}$ with each hyperplane $\{x \cdot e_{d+1} = n\}$ (for $n \in \mathbb{Z}$) is a copy of $\mathbb{G}_d$, and there are no edges in $\mathbb{G}$ connecting these hyperplanes. Therefore the distribution of $\mathbb{G}$ is invariant under translations. Furthermore, each $x \in \mathbb{Z}^{d+1}$ has out-degree one, it has no directed cycles of length at least three, and each $C_x$ is finite. Last, since the undirected version of $\mathbb{G}_d$ has multiple infinite components, so does the undirected version of $\mathbb{G}$ (in fact it has infinitely many). This proves the case $d\geq 3$ and $k=\infty$.

\subsubsection{Case $d \geq 2$ and $k=1$}\label{sec: dyadic}
For the case $d \geq 2$ and $k=1$, we give a dyadic construction. Define the orthant $\mathsf{O} = \{x \in \mathbb{Z}^d : x \cdot e_i \geq 0 \text{ for all } i \}$. For any nonzero $x \in \mathsf{O}$, let
\[
k(x) = \min\left\{k \geq 1 : x /2^k \notin \mathbb{Z}^d\right\}.
\]
Because $x/2^{k(x)-1} \in \mathbb{Z}^d$ but $x/2^{k(x)} \notin \mathbb{Z}^d$, at least one coordinate of $x/2^{k(x)-1}$ is odd. Let $i(x)$ be the largest such index. (For example, if $x = (4,8,15)$ then $k(x) = 1$ and $i(x) = 3$, and if $x=(0,8,16)$ then $k(x) = 4$ and $i(x) = 2$.) Now define $\eta_0 \in \{0,1\}^{\vec{\mathcal{E}}^d}$ by
\[
\eta_0\left( \langle x, x-e_{i(x)}\rangle\right) = 1 \text{ for all nonzero }x \in \mathsf{O},
\]
and $\eta_0\left( \langle x,y \rangle\right) = 0$ for all other directed edges $\langle x,y \rangle$. Write $\mathbb{G}_0$ for the directed graph corresponding to $\eta_0$. See Figure~\ref{fig: dyadic}.

	\begin{figure}
\hbox{\hspace{3cm}\includegraphics[width=.6\textwidth, trim={10cm 10cm 6cm 2cm}, clip]{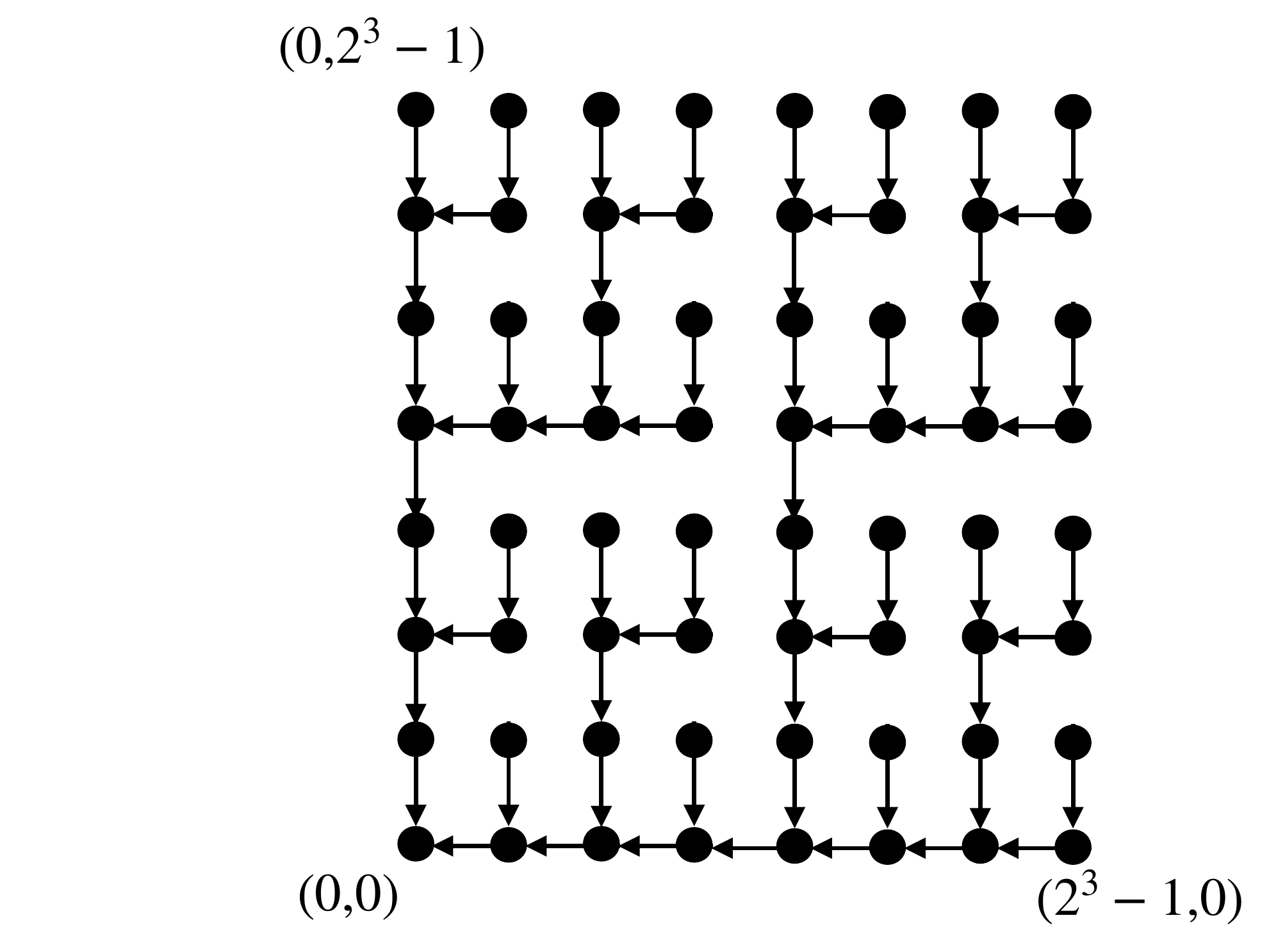}}
  \caption{Illustration of the dyadic construction of $\eta_0$ for $d=2$ restricted to the box $\mathsf{C}_3$ of side-length $8=2^3$. In each block of side-length 2, arrows point down from the top level, and to the left from the lower-right vertex. In each block of side-length 4, we repeat this construction, viewing the side-length 2 blocks as vertices: the top-level blocks point down and the lower-right block points left. In blocks of side-length 8 (the whole figure), the top-level side-length 4 blocks point down and the lower-right one points left.}
  \label{fig: dyadic}
\end{figure}

In $\mathbb{G}_0$, the vertex 0 has out-degree 0, as does every vertex that is not in $\mathsf{O}$. Each non-zero $x \in \mathsf{O}$ has out-degree one. Therefore the directed subgraph $\Gamma_x^0$ of $\mathbb{G}_0$ induced by the vertices $y$ such that $x \to y$ is a directed path. Starting from any $x \in \mathsf{O}$ and moving along $\Gamma_x^0$, the vertices are obtained from $x$ as follows. We decrement the largest odd coordinate of $x$ by 1, then the next largest odd coordinate by 1, and so on, until all coordinates are even. Then we decrement the largest coordinate that is not divisible by 4 repeatedly until it becomes divisible by 4, then the next largest coordinate that is not divisible by 4, and so on, until all coordinates are divisible by 4. We then iterate the steps with 8 in place of 4, then 16, and all powers of 2, until we reach the origin. We observe that in this procedure,
\begin{equation}\label{eq: divisibility}
\text{after all coordinates are divisible by }2^k\text{, at most one is not at any time}.
\end{equation}

For any $k \geq 1$, $\mathsf{O}$ is a disjoint union of ``$2^k$-boxes'' of the form $2^kz + \mathsf{C}_k$, where $\mathsf{C}_k = \{0, \dots, 2^k-1\}^d$ and $z \in \mathsf{O}$. Here,
\begin{equation}\label{eq: z_form}
\text{if }x \in 2^kz + \mathsf{C}_k \text{, then } z = \left( \left\lfloor (x/2^k) \cdot e_1 \right\rfloor, \dots, \left\lfloor (x/2^k) \cdot e_d \right\rfloor\right).
\end{equation}
Note that for $k \geq 1$ and $z \in \mathsf{O}$,
\begin{equation}\label{eq: k_box}
\text{if }x \in 2^kz + \mathsf{C}_k \text{ then }\Gamma_x^0 \text{ touches } 2^kz \text{ within } d2^k \text{ steps without leaving }2^kz+\mathsf{C}_k. 
\end{equation}
This follows from the above description of $\Gamma_x^0$: as we traverse $\Gamma_x^0$, we start at $x$ and decrement coordinates in the order described above until all coordinates are divisible by $2^k$. 

Before we translate and average to build a measure, we note the following properties of $\mathbb{G}_0$:
\begin{enumerate}
\item[(A)] If $\langle x,y \rangle$ is an edge of $\mathbb{G}_0$, then $\{x,y\} \in \mathcal{E}^d$. This is clear by the construction.
\item[(B)] Each $x \in \mathsf{O}$ that is nonzero has out-degree one in $\mathbb{G}_0$. This was stated above.
\item[(C)] $\mathbb{G}_0$ has no directed cycles. This is because each directed edge points in the direction of decreasing $i$-th coordinate for some $i$.
\item[(D)] If $x \in \mathsf{O}$ is in a $2^k$-box $2^kz + \mathsf{C}_k$ for $z \in \mathsf{O}$ and $k \geq 1$, and has at least two coordinates which are not multiples of $2^k$, then $C_x^0$ (the subgraph of $\mathbb{G}_0$ induced by $y$ such that $y \to x$ in $\mathbb{G}_0$) contains no vertices outside of $2^kz + \mathsf{C}_k$. To prove this, suppose that $y$ is a vertex of $C_x^0$. If $y$ is in a different $2^k$-box from $x$, say $2^kz' + \mathsf{C}_k$ for $z' \neq z$, then by \eqref{eq: k_box}, traversing $\Gamma_y^0$ leads us to $2^kz'$ without leaving $2^kz' + \mathsf{C}_k$ (in particular not touching $x$). Because of \eqref{eq: divisibility} and the assumed properties of $x$, $\Gamma_y^0$ never touches $x$.
\item[(E)] For $k \geq 1$, if $x,y$ are elements of the same $k$-box, then while traversing $\Gamma_x^0$ starting from $x$, we intersect $\Gamma_y^0$ within $d2^k$ steps. Indeed, we note that such $x$ and $y$ can be written for some $z \in \mathsf{O}$ as
\[
x = 2^kz + x',~ y = 2^kz + y' \text{ for } x',y' \in \mathsf{C}_k.
\]
So by \eqref{eq: k_box}, after at most $d2^k$ steps on $\Gamma_x^0$, we reach $2^kz$ (and similarly for $\Gamma_y^0$).
\end{enumerate}

The next step is to define a sequence of variables $(Z_n)_{n \geq 1}$ such that $Z_n$ is uniform on $\{0, \dots, 2^n-1\}^d$, and set $\eta_n = T^{Z_n}\eta_0$ to be the translation of $\eta_0$ by $Z_n$. Because the $\eta_n$ form a tight sequence (the space $\{0,1\}^{\vec{\mathcal{E}}^d}$ is compact), there is a subsequence $(n_i)$ such that $\eta_{n_i}$ converges in distribution to some $\eta$. (In fact, a subsequence is not necessary.) Letting $\mathbb{G}$ be the (random) directed graph corresponding to $\eta$, it is standard that $\mathbb{G}$ is invariant under translations. We are then left to prove that a.s., $\mathbb{G}$ satisfies the conditions of Theorem~\ref{thm: construction}, and that a.s., the undirected version of $\mathbb{G}$ has only one component.

To show the desired properties of $\mathbb{G}$, we start with item 1 of Theorem~\ref{thm: construction}, and this is the most obvious. For any $x,y \in \mathbb{Z}^d$, the event $\left\{\tau \in \{0,1\}^{\vec{\mathcal{E}}^d}  : \tau(\langle x,y \rangle) = 1\right\}$ is a cylinder event, so its indicator is a (bounded) continuous function. Therefore $\mathbb{P}(\eta(\langle x,y \rangle) = 1) = \lim_{k \to \infty} \mathbb{P}(\eta_{n_k}(\langle x,y \rangle) = 1)$. If $x$ and $y$ are not neighbors (that is, $\|x-y\|_1 > 1$), this probability is zero by item (A) above. This means $\mathbb{G}$ satisfies item 1 of Theorem~\ref{thm: construction} a.s.

For item 2, note that the event that the origin has out-degree one is a cylinder event. Again this implies that
\[
\mathbb{P}(0 \text{ has out-degree one in } \mathbb{G}) = \lim_{i \to \infty} \mathbb{P}(Z_{n_i} \text{ has out-degree one in } \mathbb{G}_0).
\]
By item (B), the right side equals $\lim_{i \to \infty} (1-1/2^{dn_i}) = 1.$ By translation invariance, we conclude item 2.
By a similar argument, we can show item 3: a.s.~$\mathbb{G}$ has no directed cycles of length at least 3. (In fact, it has no directed cycles.) Letting $\mathcal{C}$ be a deterministic (finite) directed cycle, the event that all directed edges in $\mathcal{C}$ are present in the graph (the event $\{\tau : \tau(\langle x,y \rangle) = 1 \text{ for all edges } \langle x,y \rangle \text{ in } \mathcal{C}\}$) is again a cylinder event. By item (C) above, the probability that $\eta$ is in this event is zero. Taking a union over all finite cycles shows item 3.

We now show item 4: a.s.~for each $x \in \mathbb{Z}^d$, the graph $C_x$ is finite. By translation invariance, it suffices to consider $x=0$. Writing $\# C_0$ for the number of vertices in $C_0$, note that because $\{\# C_0 > \lambda\}$ is a cylinder event, we have
\[
\mathbb{P}(\#C_0 > \lambda) = \lim_{i \to \infty} \mathbb{P}\left( \#C_{Z_{n_i}}^0 > \lambda\right).
\]
For $\lambda \geq 2^d$ and $k = \lfloor \log_{2^d} \lambda \rfloor$, write $z_{n_i}$ for the unique point of $\mathsf{O}$ such that $Z_{n_i} \in 2^k z_{n_i} + \mathsf{C}_k$. Then by item (D), 
\begin{align*}
\lim_{i \to \infty} \mathbb{P} \left( \#C_{Z_{n_i}}^0 > \lambda\right) &\leq \lim_{i \to \infty} \mathbb{P} \left( C_{Z_{n_i}} \text{ contains a vertex outside } 2^kz_{n_i} + \mathsf{C}_k\right) \\
&\leq \lim_{i \to \infty} \mathbb{P}\left( \text{at most one coordinate of }Z_{n_i} \text{ is not divisible by }2^k\right) \\
&= 2^{-dk} + d2^{-k}(1-2^{-k})^{d-1} \\
&\leq \frac{d+1}{2^k} \\
&\leq \frac{2(d+1)}{\lambda^{\frac{1}{d}}}.
\end{align*}
Letting $\lambda \to \infty$, we obtain $\#C_0 < \infty$ a.s.~and this proves item 4. 

Finally we prove that the undirected version of $\mathbb{G}$ has one infinite component a.s. For $x \in \mathbb{Z}^d$ and an integer $\lambda>0$, let $E_x = E_x(\lambda)$ be the event that there are directed paths $\pi_0$ in $\Gamma_0$ starting from $0$ and $\pi_x$ in $\Gamma_x$ starting from $x$, both with $\lambda$ many edges, which intersect. Note that this event is defined even when some vertices have out-degree greater than one, but in our graphs $\Gamma_0$ and $\Gamma_x$ are directed paths. (As usual, $\Gamma_x$ is the subgraph induced by the vertices $y$ such that $x \to y$.) Because $E_x$ is also a cylinder event,
\begin{equation}\label{eq: limit_probability}
\mathbb{P}(\eta \in E_x) = \lim_{i \to \infty} \mathbb{P}(\eta_{n_i} \in E_x).
\end{equation}
The event on the right is that $\Gamma_{Z_{n_i}}^0$ and $\Gamma_{Z_{n_i}+x}^0$ (recall these are the paths in the graph $\mathbb{G}_0$) both have at least $\lambda$ many steps, and, following either from its starting point, we intersect the other within $\lambda$ many steps. Note that if every coordinate of $Z_{n_i}$ is at least $\lambda + \|x\|_\infty$, then both paths will have at least $\lambda$ many steps. Furthermore, if both of these points are in the same $2^k$-box $2^kz_{n_i} + \mathcal{C}_k$, then by item (E), the paths will intersect within $d2^k$ steps. Therefore if we put $k = \lfloor \log_2 (\lambda/d) \rfloor$ (for $\lambda \geq 2d$), we see that if $Z_{n_i}$ and $Z_{n_i}+x$ are in the same $2^k$-box, then $\Gamma_{Z_{n_i}}^0$ and $\Gamma_{Z_{n_i}+x}^0$ will intersect within $\lambda$ many steps. This means the right side in \eqref{eq: limit_probability} is at least
\begin{align}
&\lim_{i \to \infty} \mathbb{P}\left( Z_{n_i} \text{ and } Z_{n_i}+x \text{ are in the same }2^k \text{-box and have all coordinates }\geq \lambda\right) \nonumber \\
\geq~& \lim_{i \to \infty} \left[ 1 - \frac{d(\lambda+\|x\|_\infty)}{2^{n_i}} - \mathbb{P}\left( Z_{n_i} \text{ and } Z_{n_i} + x \text{ are in different } 2^k\text{-boxes}\right)\right] \label{eq: continue_with_limit}.
\end{align}

If $Z_{n_i}$ and $Z_{n_i}+x$ are in different $2^k$-boxes, then by \eqref{eq: z_form},
\[
\left( \lfloor (Z_{n_i}/2^k) \cdot e_1 \rfloor, \dots, \lfloor  (Z_{n_i}/2^k) \cdot e_d\rfloor\right) \neq \left( \lfloor ((Z_{n_i} + x)/2^k) \cdot e_1 \rfloor, \dots, \lfloor ((Z_{n_i} + x)/2^k) \cdot e_d\rfloor\right).
\]
The probability of this is at most $2d\|x\|_\infty / 2^k$. Putting this in \eqref{eq: continue_with_limit}, and then combining with \eqref{eq: limit_probability}, we find
\[
\mathbb{P}(\eta \in E_x) \geq \lim_{i \to \infty}  \left[ 1 - \frac{d(\lambda +\|x\|_\infty)}{2^{n_i}} - \frac{2d \|x\|_\infty}{2^k}\right] \geq 1 - \frac{4d^2\|x\|_\infty}{\lambda}.
\]
Recalling the definition of $E_x$, if we take $\lambda \to \infty$, we see that a.s.~both $\Gamma_0$ and $\Gamma_x$ are infinite and intersect. Since this is true for all $x$, a.s.~the undirected version of $\mathbb{G}$ has one infinite component. This completes the proof of the case $d \geq 2$ and $k=1$.

\subsubsection{Case $d\geq 3$ and $k \in [2,\infty)$}
We have seen in Section \ref{sec: infinitely_many} that when $d \geq 3,$ it is possible to construct nearest-neighbor graphs whose undirected versions have infinitely many infinite components, contrary to the situation when $d = 2$. This could lead one to ask about the possibility of some arbitrary finite number $k \geq 2$ of components.
Fix some such integer $k  \in [2, \infty)$ for the remainder of this section. We will explicitly construct a translation-invariant measure $\mathbb{P}$ such that, a.s., $\mathcal{N}$ has exactly $k$ infinite components. We do this by describing how to generate a translation-invariant random directed graph $\mathbb{G}$ which satisfies the hypotheses of Theorem \ref{thm: construction} and whose unoriented version has exactly $k$ infinite components. We write $\mathbb{G} = (\mathbb{Z}^d, \vec{E})$; we construct the (directed) edge set $\vec{E}$ in stages, writing $\vec{E} = \vec{E}^{(1)} \cup \dots \cup \vec{E}^{(k+1)}$.

Recall the nearest-neighbor measure constructed in Section \ref{sec: dyadic}. This was a translation-invariant nearest-neighbor model with the property that, a.s., $\mathcal{N}$ has exactly one infinite component. Let $\mathbb{G}^{(1)}, \ldots, \mathbb{G}^{(k)}$ denote $k$ independent samples from this measure. We begin by using $\mathbb{G}^{(1)}$ to define $\vec{E}^{(1)}$. Given an edge $\{x, x + e_i\}$ of $\mathcal{E}^d$, we include in $\vec{E}^{(1)}$ an appropriate orientation of each of the edges $\{4k x, 4k x + e_i\}, \{4 k x + e_i, 4 k x + 2 e_i\}, \ldots, \{4 k x + (4k-1) e_i, 4k(x+e_i)\}$; exactly one orientation of each of these edges will be chosen to appear in $\vec{E}^{(1)}$. Which orientations of each of these edges appears in $\vec{E}^{(1)}$ depends on whether a) $\langle x, x+e_i\rangle$ is an edge of $\mathbb{G}^{(1)}$, b)$\langle x + e_i, x \rangle$ is an edge of  $\mathbb{G}^{(1)}$, or c) neither $\langle x, x+e_i\rangle$ nor $\langle x + e_i, x \rangle$ is an edge of $\mathbb{G}^{(1)}$ (note that, by the construction of $\mathbb{G}^{(1)}$, these are a.s.~the only possibilities).

\begin{itemize}
\item In case a), the edges $\langle 4 k x + \ell e_i, 4 k x + (\ell + 1) e_i \rangle \in \vec{E}^{(1)}$ for each $0 \leq \ell \leq 4k-1$ (and the other orientation is omitted: $\langle 4 k x + (\ell+1) e_i, 4 k x + \ell e_i \rangle \notin \vec{E}^{(1)}$).
\item Case b) is identical but reflected: each edge $\langle 4 k x + (\ell+1) e_i, 4 k x + \ell e_i \rangle \in \vec{E}^{(1)}$ for each $0 \leq \ell \leq 4k-1$, and the other orientation is again omitted.
  \item   Lastly, in case c), we orient edges toward the closer of $4 k x$ and $4k(x + e_i)$: the edges $\langle 4 k x + e_i, 4 k x\rangle,$ $\langle 4 k x + 2 e_i, 4 k x + e_i \rangle,$ \ldots, $\langle 4kx + 2 k e_i, 4 k x + (2k-1) e_i\rangle \in \vec{E}^{(1)}$, and also $\langle 4kx + (2k+1) e_i, 4kx+ (2k+2) e_i\rangle, \ldots, \langle 4kx + (4k-1) e_i, 4k(x+ e_i)\rangle \in \vec{E}^{(1)}$, with none of the reversed orientations of these edges appearing in $\vec{E}^{(1)}$.
  \end{itemize}
  We can think of this construction as in a sense ``stretching out the lattice $\mathbb{Z}^d$'' by a factor $4k$; each edge of the lattice is turned into a segment of $4k$ edges. The above definition guarantees that these segments are traversed by an oriented path in $\vec{E}^{(1)}$ exactly when the corresponding ``un-stretched'' edges of $\vec{\mathcal{E}}^d$ appear in $\mathbb{G}^{(1)}$.

Let $V^{(1)}$ denote the set of endpoints of the edges considered above --- in other words,
\[V^{(1)} = \{4k x + \ell e_i: \, x \in \mathbb{Z}^d,\, 0 \leq \ell \leq 4k-1,\, 1 \leq i \leq d\}. \]
We can consider $\vec{E}^{(1)}$ as inducing a random directed graph with vertex set $V^{(1)}$. We note several properties of this graph which follow directly from the definition and from properties of $\mathbb{G}^{(1)}$. First,
\begin{equation}
  \label{eq:outdegthreed}
  \text{a.s., each $y \in V^{(1)}$ has out-degree one in $\vec{E}^{(1)}$;}
\end{equation}
Next, for each $x \in \mathbb{Z}^d$,
  \begin{align}\label{eq:backbd}
    \#\{y \in V^{(1)}: \, y \to 4 k x \text{ by a path in $\vec{E}^{(1)}$}\} \leq 8 k d \# \{z \in \mathbb{Z}^d: \, z\to x \text{ by a path in $\mathbb{G}^{(1)}$}  \}.
  \end{align}
From \eqref{eq:backbd}, we immediately see that
\begin{equation}
  \label{eq:backbd2}
  \text{a.s., for each $y \in V^{(1)}$,} \quad \# \{z \in V^{(1)}: \, z \to y \text{ by a path in $\vec{E}^{(1)}$}\} < \infty.
\end{equation}
No edge of $\vec{E}^{(j)},$ $j \geq 2$ will be incident to any vertex of $V^{(1)}$, so the above properties will be preserved throughout the remainder of the construction. Statements \eqref{eq:outdegthreed} and \eqref{eq:backbd2} (and their analogues for the endpoints of edges in $\vec{E}^{(j)},$ $j \geq 2$) will guarantee that the graph $\mathbb{G}$ satisfies the hypotheses of Theorem \ref{thm: construction} and hence can be represented as a nearest neighbor graph.

Let $E^{(1)}$ denote the set of undirected versions of edges in $\vec{E}^{(1)}$; a final important property of the above is that
\[(V^{(1)}, E^{(1)}) \text{ a.s.~has exactly one infinite component}. \]
This is again easy to see from the definition. Indeed, the construction above preserves the component structure of vertices of the form $4 k x$: there is a path of edges of $E^{(1)}$ from $4 k x_1$ to $4 k x_2$ if and only if there is a path of (undirected versions of) edges of $\mathbb{G}^{(1)}$ from $x_1$ to $x_2$, and other vertices of $V^{(1)}$ can only connect up via vertices of the form $4 k x$.

The construction of $\vec{E}^{(2)}, \ldots, \vec{E}^{(k)}$ proceeds analogously, but on shifted sublattices. The vertex set $V^{(2)} = V^{(1)} + (4, 4, \ldots, 4)$.
Edges are included in $\vec{E}^{(2)}$ via an analogous version of the above procedure, but using the realization $\mathbb{G}^{(2)}$ and with the entire construction shifted by $(4, \ldots, 4)$; we use the status of edges of the form $\langle x, x + e_i\rangle$ to determine the status of directed versions of edges of the form $\{4 k x + (4, 4, \ldots, 4) + \ell e_i, 4  k x + (4, 4, \ldots, 4) + (\ell + 1) e_i\}$. To construct $\vec{E}^{(3)}$, we proceed analogously but with $V^{(3)} = V^{(1)} + (8, 8, \ldots, 8)$, and so on. A vertex $x \in V^{(j)}$ has the property that
\begin{equation}
 \label{eq:coordmod}
 x \cdot e_i \equiv 4(j-1) \, \mod 4k
 \end{equation}
for at least $d-1$ values of $i$; hence,  $V^{(j)} \cap V^{(k)} = \emptyset$ when $j \neq k$. In other words, we have constructed $k$ ``noninteracting'' and independent directed graphs on distinct sublattices of $\mathbb{Z}^d$, each of which obeys the properties \eqref{eq:outdegthreed} and \eqref{eq:backbd}, and whose undirected version a.s.~has exactly one infinite component.

To complete the construction, we must choose $\vec{E}^{(k+1)}$ in a way which guarantees the hypotheses of Theorem \ref{thm: construction} are satisfied. Once this is done, we will have constructed a random graph which is realizable as a nearest-neighbor graph, though (as in the construction in Section \ref{sec: dyadic}) this nearest-neighbor model will not be translation-invariant. To finish the construction and recover translation invariance, we conclude by shifting the entire graph by an independent random integer vector in the cube $[0, 4k-1)^d$.

It remains to choose $\vec{E}^{(k+1)}$.  Suppose we chose $\vec{E}^{(k+1)} = \emptyset$, or in other words had the (oriented) edge set of $\mathbb{G}$ be $\vec{E}^{(1)}\cup \dots \cup \vec{E}^{(k)}$. Then the hypotheses of Theorem \ref{thm: construction} would not be satisfied: the set $V^{(k+1)}:= \mathbb{Z}^d \setminus (V^{(1)} \cup \dots \cup V^{(k)})$ is nonempty, and each vertex of $V^{(k+1)}$ would have both out- and in-degree zero under this oriented edge set. We complete the construction by choosing $\vec{E}^{(k+1)}$ in a way such that
\begin{itemize}
\item Each vertex of $V^{(k+1)}$ has out-degree one in $(V^{(k+1)},\vec{E}^{(k+1)})$;
\item No edge of $\vec{E}^{(k+1)}$ connects a vertex of $V^{(k+1)}$ to a vertex of $\mathbb{Z}^d \setminus V^{(k+1)}$ (or vice-versa);
  \item The undirected version $E^{(k+1)}$ of $\vec{E}^{(k+1)}$ has components of $\ell^{\infty}$ diameter no larger than $4k$.
  \end{itemize}
  These properties guarantee that the directed graph $(V^{(k+1)}, \vec{E}^{(k+1)})$ again ``does not interact with'' the graphs $(V^{(i)}, \vec{E}^{(i)})$ for $i \leq k$, and that the graph $\mathbb{G}$ satisfies the hypotheses of Theorem \ref{thm: construction}. Moreover, since there are no infinite components in $(V^{(k+1)}, \vec{E}^{(k+1)})$, the graph $\mathbb{G}$ has exactly $k$ infinite components.

  We choose the edges of  $\vec{E}^{(k+1)}$ to have both endpoints in a common cube of the form $4 k x + [-2 k, 2k)^d$; this will guarantee the diameter condition above holds. To do this, for each $x \in \mathbb{Z}^d$, consider the site-components of $R_x:=  \left(4 k x + [-2 k, 2k)^d\right) \setminus (V^{(1)} \cup \ldots \cup V^{(k)})$.  For each site-component $\mathcal{C}$, choose a deterministic spanning tree of the vertices of $\mathcal{C}$ oriented toward some deterministic root, then insert a single additional oriented edge from this root toward one of its neighbors (note that this choice can be done in a non-random way, identically for each $x$; all of the $R_x$'s are translates of one another).

  This construction guarantees that the above bulleted properties hold: the latter two are obvious, and the first holds by construction as long as each site-component $\mathcal{C}$ has at least two vertices.
  To see why each site-component has at least two vertices, consider a vertex $y \in R_0$, and let $\mathcal{C}$ be the corresponding site component of $y$. By symmetry, we may assume $y \cdot e_1 \geq 0$. If $y + e_1 \in \mathcal{C}$, then we are done. Otherwise, if $y + e_1 \notin \mathcal{C}$, then  either i) $y  + e_1 \notin [-2k, 2k)^d$ or ii) $y + e_1 \in V^{(j)}$ for some $1 \leq j \leq k$. In case i), we have $y \cdot e_1 = 2k$, and then $(y - e_1) \cdot e_1 = 2k-1$, whence $y-e_1$ cannot be in $V^{(m)}$ for any $1 \leq m \leq k$. In case ii), $y - e_1 \in [-2k, 2k)^d$  since $y \cdot e_1 \geq 0$. Moreover, $y - e_1 \notin V^{(j)}$, since otherwise we would also have $y \in V^{(j)}$. Lastly, $y-e_1$ cannot be in $V^{(m)}$ for any $m \neq j$, because then $V^{(j)}$ and $V^{(m)}$ would be at Euclidean distance two from each other. We conclude that in either case, $y - e_1$ is also an element of $y$'s site component $\mathcal{C}$. Thus, the site components of $R_x$ are not singletons, and so the first bulleted property above holds.

\bigskip
\noindent
{\bf Acknowledgements.} The research of M. D. is supported by an NSF CAREER grant. The research of J. H. is supported by NSF grant DMS-161292, and a PSC-CUNY Award, jointly funded by The Professional Staff Congress and The City University of New York. M. D. thanks A. Krishnan for pointing out reference \cite{ZernerMerkl}.

\end{document}